\documentclass[12pt]{amsart}

\usepackage{t1enc}
\usepackage[latin1]{inputenc}
\usepackage[english]{babel}

\usepackage{amsthm}
\usepackage{yfonts}

\usepackage{bbm}
\usepackage{bm}
\usepackage{mathrsfs}

\usepackage{amsmath,amssymb,enumerate,color,graphicx}
\usepackage{tikz-cd}
\usepackage[pdfpagelabels=true]{hyperref}

\usepackage[capitalise]{cleveref}
\crefname{theorem}{Theorem}{Theorems}
\crefname{lemma}{Lemma}{Lemmas}
\crefname{corollary}{Corollary}{Corollaries}
\crefname{proposition}{Proposition}{Propositions}
\crefname{conjecture}{Conjecture}{Conjectures}
\crefname{question}{Question}{Questions}
\crefname{definition}{Definition}{Definitions}
\crefname{example}{Example}{Examples}
\crefname{remark}{Remark}{Remarks}
\crefname{enumi}{}{}
\crefname{equation}{}{}

\newtheorem{theorem}{Theorem}[section]
\newtheorem{corollary}[theorem]{Corollary}
\newtheorem{lemma}[theorem]{Lemma}
\newtheorem{proposition}[theorem]{Proposition}

\theoremstyle{definition}

\theoremstyle{remark}
\newtheorem{remark}[theorem]{Remark}

\newcommand{\al}{\alpha}
\newcommand{\be}{\beta}
\newcommand{\ga}{\gamma}

\newcommand{\cC}{\mathcal{C}}

\newcommand{\beq}{\begin{equation}}
\newcommand{\eeq}{\end{equation}}

\makeatletter
\newsavebox\myboxA
\newsavebox\myboxB
\newlength\mylenA

\newcommand*\xoverline[2][0.9]{%
    \sbox{\myboxA}{$\m@th#2$}%
    \setbox\myboxB\null
    \ht\myboxB=\ht\myboxA%
    \dp\myboxB=\dp\myboxA%
    \wd\myboxB=#1\wd\myboxA
    \sbox\myboxB{$\m@th\overline{\copy\myboxB}$}
    \setlength\mylenA{\the\wd\myboxA}
    \addtolength\mylenA{-\the\wd\myboxB}%
    \ifdim\wd\myboxB<\wd\myboxA%
       \rlap{\hskip 0.5\mylenA\usebox\myboxB}{\usebox\myboxA}%
    \else
        \hskip -0.5\mylenA\rlap{\usebox\myboxA}{\hskip 0.5\mylenA\usebox\myboxB}%
    \fi}

    \newcommand*\xoverlinew[2][1]{%
    \sbox{\myboxA}{$\m@th#2$}%
    \setbox\myboxB\null
    \ht\myboxB=\ht\myboxA%
    \dp\myboxB=\dp\myboxA%
    \wd\myboxB=#1\wd\myboxA
    \sbox\myboxB{$\m@th\overline{\copy\myboxB}$}
    \setlength\mylenA{\the\wd\myboxA}
    \addtolength\mylenA{-\the\wd\myboxB}%
    \ifdim\wd\myboxB<\wd\myboxA%
       \rlap{\hskip 0.5\mylenA\usebox\myboxB}{\usebox\myboxA}%
    \else
        \hskip -0.5\mylenA\rlap{\usebox\myboxA}{\hskip 0.5\mylenA\usebox\myboxB}%
    \fi}
		
\makeatother

\newcommand{\oX}{\skew{3}{\xoverline}{X}}

\newcommand{\oU}{\skew{3}{\xoverlinew}{U}}

\newcommand{\oW}{\skew{2}{\xoverlinew}{W}}

\newcommand{\oJ}{\skew{3.8}{\xoverline}{J}}

\begin{document}
\title[Geodesics in Conformally Compact Manifolds]{Asymptotic behavior of geodesics\\ in conformally compact manifolds}

\author{Sean N. Curry \& Achinta K. Nandi}


\begin{abstract} 
We study the asymptotic behavior of geodesics near the boundary of a conformally compact Riemannian manifold $(X,g)$. In the case where the sectional curvature at infinity is constant (the asymptotically hyperbolic case) it is known that non-trapped geodesics extend to the conformal infinity as smoothly immersed curves in $\oX$ and that the conformal infinity can be locally smoothly parametrized by the initial directions of such non-trapped geodesics starting at a given point. We show that in the general case non-trapped geodesics typically only extend to the conformal infinity with $\cC^{1,\alpha}$ regularity, due to the presence of a logarithmic singularity sourced by the gradient of the limiting sectional curvature.  In spite of this singular behavior, we show that the endpoints of such geodesics depend smoothly on the initial conditions, so that the conformal infinity can still be locally smoothly parametrized using the initial velocities of geodesics. 
In particular,  the smooth structure of the conformal infinity is canonically determined by $(X,g)$. 
We also show how to specify `initial data' for geodesics at the conformal infinity and obtain an asymptotic expansion that parametrizes the geodesics tending toward a given boundary point.
\end{abstract}


\maketitle


\section{Introduction and Main Results}

The notion of conformal compactification of semi-Riemannian metrics was introduced by Penrose for the study of asymptotic properties of radiative fields in general relativity \cite{Penrose1963,Penrose1965}. Although Penrose was initially interested in the case where the curvature tends to zero at infinity, the case where the curvature is negative near infinity has also proved to be of great interest. In the Lorentzian setting this gives a class of spacetimes that are asymptotically similar to anti-de Sitter space, and in the Riemannian setting a class of Riemannian manifolds whose asymptotic geometry is similar to that of hyperbolic space. An early application of this class of geometries can be found in the work of Hawking and Page on black hole thermodynamics \cite{HawkingPage1982}, but the development of this area really began with the seminal work of Fefferman and Graham on conformal invariants \cite{FeffermanGraham1984,FeffermanGraham2012}. Starting in the late 1990s, an explosion of activity was sparked by the work of Maldacena, Witten and others on the AdS/CFT correspondence in physics \cite{Maldacena1998,Witten1998,GKP1998,HenningsonSkenderis1998}, and the area continues to be highly active today.

In this paper we study the boundary behavior of geodesics in a complete Riemannian manifold $(X,g)$ of dimension $n+1$ that is \textit{conformally compact} in the sense that $X$ is the interior of a compact manifold $\oX$ with boundary $\partial X$ and $g$ is of the form $\rho^{-2}h$, where $\rho$ is a defining function for $\partial X$ (with $\rho>0$ on $X$) and $h$ is a metric on $\oX$ that is smooth and nondegenerate up to the boundary.  We are particularly concerned with the regularity of the geodesics at the boundary and the smoothness of the dependence of the boundary endpoints on the initial direction of the geodesic. These questions are well understood in the \textit{asymptotically hyperbolic} case, where the limiting sectional curvature of $g$ is constant on $\partial X$ \cite{Mazzeo-Thesis}; it was already noted in \cite{Mazzeo-Thesis}, however, that when the asymptotic hyperbolicity condition fails there is an additional singularity present at the boundary in the (rescaled) geodesic equations and this presents an obstacle to generalizing beyond the asymptotically hyperbolic setting. The main goal of the present paper is to overcome this obstacle. We find, moreover, that new phenomena arise when the metric $g$ fails to be asymptotically hyperbolic.

If $(X,g)$ is conformally compact in the sense just defined, with $g=\rho^{-2}h$, then it is easy to show that the sectional curvatures of $g$ all tend towards $-|d\rho|_h^2$ as one approaches a point of the conformal infinity $\partial X$ (and hence are uniformly negative near $\partial X$). It follows that $(X,g)$ is asymptotically hyperbolic if and only if the function $\kappa =|d\rho|_h$ on $\partial X$ is constant. Our investigation shows that, while in the asymptotically hyperbolic case all geodesics for $g$ near $\partial X$ extend (by appending their boundary endpoints) to give smoothly embedded curves in $\oX$, geodesics that tend towards a point in $\partial X$ with $d\kappa\neq 0$ pick up a weak singularity at that point arising from the singularity in the geodesic equations sourced by  $d\kappa$ and are only $\cC^{1,\alpha}$ up to the boundary in general. Nevertheless, we find that the dependence of the endpoint of geodesics on the initial point and direction continues to be $\cC^{\infty}$-smooth even when the asymptotic hyperbolicity condition fails.

While the questions addressed in this paper are of intrinsic interest geometrically, our initial motivation for the present work came from complex analysis in several variables.  
In terms of the behavior of the geodesic equations, generalizing from asymptotically hyperbolic to merely conformally compact manifolds turns out to be highly analogous to generalizing from the case where the log-term in the Bergman kernel is absent to the (generic) case where it is present in Fefferman's treatment of the boundary behavior of geodesics of the Bergman metric on strongly pseudoconvex domains in \cite{Fefferman1974} (even though the presence of the log-term does not correspond to a failure of asymptotic \textit{complex} hyperbolicity  \cite{EpsteinMelroseMendoza1991}). A major motivation for the present work was that it gives a simpler setting in which to consider boundary regularity questions for the cogeodesic flow than that of asymptotically complex hyperbolic metrics (or, more generally, of $\Theta$-metrics \cite{EpsteinMelroseMendoza1991}) where similar issues arise. Fefferman's treatment of the boundary behavior of geodesics of the Bergman metric on strongly pseudoconvex domains was a key step in his proof of the smoothness up to the boundary of biholomorphic mappings between such domains; this result was subsequently proved by other methods, which have been greatly developed (see, e.g., \cite{BellLigocka1980,Bell1981,Forstneric1989,Forstneric1993,PinchukShafikovSukhov2023}), but many open questions remain concerning the boundary regularity of proper holomorphic mappings in positive codimension (for example, when $N>n>1$, is a proper holomorphic map $\mathbb{B}^n \to \mathbb{B}^N$ that is, say, $\cC^2$ up to the boundary forced to be $\cC^{\infty}$ up to the boundary?). In a planned follow-up the authors intend to place the work of Fefferman \cite{Fefferman1974} in a more general setting that includes, e.g., the canonical complete K\"ahler-Einstein metrics on strongly pseudoconvex domains in $\mathbb{C}^n$ \cite{Fefferman1976,ChengYau1980}, with a view to applications to boundary regularity of proper holomorphic mappings in several complex variables. For such applications one expects that it will be necessary to be flexible with the choice of metric(s) used (unlike in the case of biholomorphic mappings, one cannot hope to use a ``canonical'' metric on both the domain and target), hence the desire to consider a more general family of metrics.

Of course, conformally compact Riemannian metrics are of considerable interest in their own right and the study of the cogeodesic flow near infinity is well motivated from the point of view of spectral and scattering theory and microlocal analysis for the Laplacian and related operators on such manifolds (the cogeodesic flow being precisely the bicharacteristic flow of the Laplacian on the unit sphere subbundle in the cotangent bundle). This was the original motivation for the consideration of the geodesic flow on conformally compact manifolds in \cite{Mazzeo-Thesis} (where the Hodge Laplacian on $q$-forms is considered, cf.\ \cite{Mazzeo1988}). 
Although some analytic results can be obtained in the general conformally compact case without extra difficulty \cite{Mazzeo-Thesis,Mazzeo1988,Mazzeo1991}, many things simplify when the limiting sectional curvature $-\kappa^2$ is constant on $\partial X$ and the asymptotically hyperbolic case has received considerably more attention than the general case (see, e.g., \cite{Chen2018,ChenHassell-I,ChenHassell-II,ChenHassell2020,EptaminitakisGraham2021,GGSU2019,Guillarmou2005,Lee1995,MazzeoMelrose1987,Vasy2013}). The nonconstancy of the function $\kappa=|d\rho|_h$ on $\partial X$ complicates both the analysis of the cogeodesic flow near the boundary \cite{Mazzeo-Thesis} and of related geometric operators such as the Laplacian $\Delta_g$ \cite{Borthwick2001,CharalambousRowlet2024,SaBarretoWang2016} (since, e.g., for $\lambda>0$ the indicial roots of $\Delta_g-\lambda$, viewed as a degenerate elliptic operator on $\oX$, are constant on $\partial X$ if and only if $\kappa$ is constant on $\partial X$). Although we only address the cogeodesic flow in this paper, our study demonstrates the significance of the gradient of the limiting sectional curvature in the asymptotic geometry and analysis of conformally compact manifolds and provides key tools for the study of geometric inverse problems in this setting (cf., e.g., \cite{GGSU2019,dHIKM-arxiv2024}).

Motivation for dropping the asymptotic hyperbolicity condition is also provided by recent interest in submanifolds in Poincar\'e-Einstein manifolds with prescribed boundary at infinity (see, e.g., \cite{AlexakisMazzeo2010,AlexakisMazzeo2015,GrahamReichert2020,Tyrrell2023,Marx-Kuo2025, Marx-Kuo-arxiv2024,CaseGrahamKuoTyrrellWaldron2025}), which is again driven by connections with physics \cite{GrahamWitten1999,RyuTakayanagi2006}. Poincar\'e-Einstein manifolds are necessarily asymptotically hyperbolic and, when sufficiently regular up to the conformal infinity, so are their minimal submanifolds. More general submanifolds, however, are typically only conformally compact (again assuming that they extend sufficiently smoothly up to the conformal infinity); indeed, a simple calculation shows that a submanifold of an asymptotically hyperbolic manifold that is $\cC^2$ up to the conformal infinity will be asymptotically hyperbolic if and only if the length of the mean curvature vector tends to a constant at the conformal infinity. Since geodesics are the $1$-dimensional minimal submanifolds, the work below can also be seen as a first step in generalizing previous work on minimal submanifolds in Poincar\'e-Einstein  manifolds (and ``partially even'' asymptotically hyperbolic manifolds more generally) to the setting of a general conformally compact ambient space; cf.\ Remark \ref{rem:minimal} below.

\subsection{Main Results}

Before stating our main results, we briefly recall some basic facts about conformally compact manifolds. We note first that when $(X,g)$ is a conformally compact Riemannian manifold with $g$ expressed as $\rho^{-2}h$, the defining function $\rho$ and the metric $h$ on $\oX$ are not canonical; if we replace $\rho$ by $\hat{\rho}=\Omega\rho$ and $h$ by $\hat{h}=\Omega^2 h$ where $\Omega>0$ is any smooth function on $\oX$, then $g=\hat{\rho}^{-2}\hat{h}$. Nevertheless, on $\partial X$ the quantity $\kappa=|d\rho|_h$ is well defined, independent of the particular choice of $h$ and $\rho$. Indeed, the quantity $-\kappa^2$ gives the limiting value of the sectional curvatures of $g$ at any point on $\partial X$, as can easily be seen by computing the Riemannian curvature tensor of $g$ in terms of that of $h$ using the standard conformal transformation law. One finds
\beq\label{eq:curvature-asymp}
\begin{aligned}
R_{ijkl} &=-\rho^{-4}|d\rho|_h^2(h_{ik}h_{jl} - h_{jk}h_{il}) + O(\rho^{-3}) \\
&= - \kappa^2(g_{ik}g_{jl} - g_{jk}g_{il}) + O(\rho^{-3}),
\end{aligned}
\eeq
where the $O(\rho^{-3})$ term is $\rho^{-3}$ times a tensor that is smooth up to the boundary of $\oX$. 

Of course, the metric $h$ (or $\hat{h}$) determines a canonical conformal structure on $\oX$ and hence also on $\partial X$. In the case when $\kappa\equiv 1$, it is standard to normalize a defining function $\rho$ (equivalently, a choice of $h$ on $\oX$) near $\partial X$ in terms of a choice of metric in the conformal class on $\partial X$ by requiring that $h$ extends this metric and $|d\rho|_{h}=1$ in a neighborhood of $\partial X$; in particular, it is natural to assume in this case that $h$ and $\rho$ have been chosen so that (near $\partial X$) $\rho$ is the $h$-distance from the boundary. 
In the general conformally compact case, however, it is not possible to normalize $h$ and $\rho$ in this way. Given a choice of $h$ and $\rho$ such that $g=\rho^{-2}h$ we then have two natural defining functions for $\partial X$, namely, $\rho$ and the $h$-distance from the boundary, which we will denote by $x$. The function $x$ is defined near $\partial X$ and we may use the flow of $\mathrm{grad}_h \hspace{0.8pt}x$ to identify a (closed) neighborhood $\oU$ of $\partial X$ in $\oX$ with the product $[0,\delta]\times \partial X$ so that $x$ becomes the coordinate on the first factor. Representing a point in $\oU\subset \oX$ by $(x,x')\in [0,\delta]\times \partial X$ our two defining functions are related by
\beq \label{eq:rho-vs-x}
\rho(x,x') = \kappa(x')x + O(x^2),
\eeq
where $\kappa= |d\rho|_h$ on $\partial X$. This relationship will play an important role in our analysis.

Since it is convenient to have a coordinate that increases as one moves towards the boundary, in the following we typically make use of $x^0=-x$ rather than $x$. Fixing a local coordinate system $(x^1,\ldots,x^n)$ for $\partial X$ we obtain a Fermi coordinate chart $(x^0,x^1,\ldots,x^n)$ for $\partial X\subset \oX$ with respect to the metric $h$. We say that a curve $\gamma: [0, \infty) \to X$ \emph{leaves all compact sets} if for every compact subset $K$ of $X$ there is $T>0$ such that $\gamma([T,\infty))\subset X\setminus K$.
Our first main result is:

\begin{theorem}\label{thm:asymptotic-behavior-of-geodesic}
Let $(X, g)$ be a conformally compact Riemannian manifold with $g= \rho^{-2}h$. Let $\gamma: [0, \infty) \to X$ be a geodesic for $g$ that leaves all compact sets. Then
\begin{itemize}
      \item[(i)] $\gamma(t)$ tends to a definite point $\lim_{t \to \infty} \gamma(t) = \gamma_{\infty} \in \partial X$;
      \item[(ii)] in a Fermi coordinate chart $(x^0,x^1, \dots, x^n)$ for the boundary $\partial X$ with respect to the metric $h$, with $x^0<0$ in X, near $\gamma_{\infty}$ we have
      $$
      \begin{aligned}
      \dot{\gamma}^{0}(t) &= \rho + O(\rho^2), \\ 
      \dot{\gamma}^{\alpha}(t) &= O(\rho^2 \log \rho) \text{ for } \alpha=1,\ldots, n;
      \end{aligned}
      $$
      \item[(iii)] $\gamma([0,\infty)) \cup \{\gamma_{\infty}\}$ is a $\cC^{1,\alpha}$-immersed curve in $\oX$ that meets the boundary orthogonally;
      \item[(iv)] the point $\gamma_{\infty} = \lim_{t \to \infty}  \gamma(t) \in \partial X$ varies smoothly as we vary the initial direction of $\gamma$ in $TX$.
\end{itemize}
\end{theorem}
\begin{remark}\label{rem:polyhom}
In the above theorem  it is meant to be understood that $h$ is a metric on $\oX$ that is smooth up to $\partial X$. Note, however, that the result still holds if $h$ is only assumed to be smooth on $X$, provided $h$ is $\cC^{1,\alpha}$ up to $\partial X$ and \emph{polyhomogeneous} in the sense of having an asymptotic expansion
$
h = \sum_{j=0}^{\infty} \sum_{i\geq m_j} h_{(i,j)}\rho^i(\log \rho)^j 
$
with smooth coefficients $h_{(i,j)}$ near $\partial X$, where $(m_j)_{j=0}^{\infty}$ is a strictly increasing sequence of integers (and requiring $h$ to be $\cC^{1,\alpha}$ forces $m_1$ to be at least $2$). Metrics with this type of regularity can be found, e.g., in \cite{CDLS2005,GrahamLee1991}. This remark applies equally to \cref{thm:boundary-exponential,cor:smooth-boundary-extension,thm:shooting-geodesics-from-boundary} below, and to the equivalence between (i) and (ii) in \cref{thm:AH-characterization}.
\end{remark}
\begin{remark}  
The conclusions of \cref{thm:asymptotic-behavior-of-geodesic} still hold if $\gamma$ is a timelike geodesic in a Lorentzian manifold with spacelike conformal infinity (and the previous remark concerning the regularity of $h$ still applies). For spacelike geodesics in a Lorentzian manifold with timelike conformal infinity one can also easily modify the arguments of this paper to show that (ii)--(iv) hold, assuming (i). 
\end{remark}

The results of \cref{thm:asymptotic-behavior-of-geodesic} are obvious in the model case of hyperbolic space $(\mathbb{B}^{n+1},g)$ and were known previously in the asymptotically hyperbolic case (as were (i) and the fact that the direction of $\dot{\gamma}(t)$ tends to that of the outward normal to $\partial X$ in the general case),  see \cite{Mazzeo-Thesis}.
Items (ii), (iii) and (iv), however, are new in the general case where $|d\rho|_h$ is not required to be constant on $\partial X$ and their proof in this case is more subtle due to a singularity in the (suitably reparametrized) geodesic equations at the boundary which is sourced by the gradient of the function $\kappa=|d\rho|_h$ on $\partial X$.
In the asymptotically hyperbolic case, the singularity is not present and the $O(\rho^2 \log \rho)$ bound in \cref{thm:asymptotic-behavior-of-geodesic}(ii) can be improved to $O(\rho^2)$. It turns out that this gives a characterization of asymptotically hyperbolic metrics among conformally compact metrics on $X$: 
\begin{theorem}\label{thm:AH-characterization}
Let $(X, g)$ be a conformally compact Riemannian manifold with $g= \rho^{-2}h$.  Suppose $\partial X$ is connected. Then the following are equivalent:
\begin{itemize}
    \item[(i)] the metric $g$ on $X$ is asymptotically hyperbolic;
    \item[(ii)] for every geodesic $\gamma: [0, \infty) \to X$ for $g$ that tends to a point $\gamma_{\infty}$ on $\partial X$, in a Fermi coordinate chart $(x^0,x^1, \dots, x^n)$ for the boundary $\partial X$ with respect to the metric $h$, near $\gamma_{\infty}$ we have
    $$
    \dot{\gamma}^{\alpha}(t) = O(\rho^2) \text{ for } \alpha =1,\ldots, n;
    $$
    \item[(iii)] for every geodesic $\gamma: [0, \infty) \to X$ for $g$ that tends to a point $\gamma_{\infty}$ on $\partial X$, $\gamma([0,\infty)) \cup \{\gamma_{\infty}\}$ is a $\cC^{\infty}$-immersed curve in $\oX$.
\end{itemize}
\end{theorem}
\begin{remark}
    In fact, for (iii) we need only require that for each point in $\partial X$ there is one geodesic $\gamma$ for $g$ tending to this point such that  $\gamma([0,\infty)) \cup \{\gamma_{\infty}\}$ is $\cC^{\infty}$; see \cref{thm:shooting-geodesics-from-boundary} and \cref{eq:obstruction} below. A similar comment applies to (ii); see \cref{rem:x-alpha-dot-more-precise-asymp}. \cref{thm:AH-characterization} also holds locally in the sense that $g$ being asymptotically hyperbolic up to some open set in $\partial X$ is equivalent to both (ii) and (iii) for geodesics that tend to points in that open set (and again one only needs (ii) or (iii) to hold for at least one geodesic tending to each point).
\end{remark}

The key observation behind the above two theorems is that, while in the asymptotically hyperbolic case the cogeodesic flow for geodesics approaching $\partial X$ can be reparametrized and extended smoothly to $\partial X$ in a certain sense \cite{Mazzeo-Thesis,Borthwick2001,GGSU2019}, in the general case one can only obtain an analogous extended flow with coefficients that are smooth on $X$ and on $\partial X$ but have a logarithmic-type conormal singularity at $\partial X$ (meaning regularity is only lost when one takes derivatives in the normal direction at $\partial X$); this singularity is absent precisely when $\kappa$ is locally constant. The conormal nature of the singularity means that, although the geodesics pick up a weak singularity at $\partial X$ (being only $\cC^{1,\alpha}$ and not $\cC^2$ there in general), the boundary endpoint $\gamma_{\infty}$ of the geodesic still varies smoothly with the initial condition (\cref{thm:asymptotic-behavior-of-geodesic}(iv)). 

Since the sectional curvatures of a conformally compact Riemannian manifold $(X,g)$ tend to $-\kappa^2<0$ on $\partial X$, they are negative and bounded away from zero there. From part (iv) of \cref{thm:asymptotic-behavior-of-geodesic} we therefore obtain the following:

\begin{theorem}\label{thm:boundary-exponential}
    Let $(X, g)$ be a conformally compact Riemannian manifold. Then
    \begin{itemize}
        \item[(i)] if $p \in X$ is sufficiently close to $\partial X$ then there is an open subset $V_p$ in the unit tangent space $S_p X \subset T_p X$ such that all geodesics in $(X,g)$ with initial direction in $V_p$ tend to a point in the boundary $\partial X$;
        \item[(ii)] the map $\exp_{p,\infty}: V_p \to \partial X$ that sends the initial velocity of a geodesic to the endpoint of that geodesic in $\partial X$ is a local $\cC^{\infty}$-diffeomorphism onto its image;
        \item[(iii)] for every point $p_{\infty} \in \partial X$ there is a point $p \in X$ and an open subset $V_{p,p_{\infty}} \subset V_p$ such that $\exp_{p, \infty}$ is a $\cC^{\infty}$-diffeomorphism from $V_{p,p_{\infty}}$ to an open neighborhood of $p_{\infty}$ in $\partial X$.
    \end{itemize}
\end{theorem}

The above result is again obvious in the model case of hyperbolic space $(\mathbb{B}^{n+1},g)$ and was known in the case where $(X,g)$ is asymptotically hyperbolic \cite{Mazzeo-Thesis}. The key point in the above theorem is that we obtain $\cC^{\infty}$-smoothness in the conformally compact case in spite of the fact that the geodesics $\gamma$ used to construct the exponential map only extend to $\cC^{1,\alpha}$-immersed curves in $\oX$ in general.

Note that \cref{thm:boundary-exponential} implies the existence of a smooth atlas for $\partial X$ determined intrinsically by the geometry of $(X,g)$. Putting this result in a broader context we note that by using the limiting behavior of geodesics a complete noncompact Riemannian manifold $(X,g)$ with sectional curvature $K\leq -\kappa_1^2 <0$ may be canonically endowed with a boundary $\partial X$ at infinity in a $\cC^0$ fashion \cite{EberleinO'Neill1973} and that when $-\kappa_2^2 \leq K\leq -\kappa_1^2 <0$ this boundary has an intrinsically determined $\cC^{\alpha}$-structure, for any $\alpha \in (0,\kappa_1/\kappa_2)$ \cite{AndersonSchoen1985} (of course, one only really needs the sectional curvature bounds near infinity for these results). 

From the point of view of our intended applications to boundary regularity problems, the smoothness of the local parametrizations of $\partial X$ obtained in \cref{thm:boundary-exponential} is key. In particular,  \cref{thm:boundary-exponential} has the following corollary:

\begin{corollary}\label{cor:smooth-boundary-extension}
    Let $(X_1,g_1)$ and $(X_2, g_2)$ be conformally compact Riemannian manifolds. Let $F$ be an isometry from $(X_1, g_1)$ to $(X_2, g_2)$. Then $F$ extends continuously to a homeomorphism from $\oX_1$ to $\oX_2$ whose restriction to the boundary gives a $\cC^{\infty}$ conformal diffeomorphism $f:\partial X_1 \to \partial X_2$.
\end{corollary}

\cref{cor:smooth-boundary-extension} is analogous to the result of Fefferman \cite{Fefferman1974} stating that a biholomorphism $F$ between smoothly bounded strongly pseudoconvex domains (which is necessarily an isometry for the respective Bergman metrics of the domains) extends continuously to a $\cC^{\infty}$ CR diffeomorphism $f$ between the boundaries. In that case the smoothness of the CR map $f$ on the boundary implies the smoothness of $F$ up to the boundary (by the Hartogs phenomenon); in other words, in the complex setting one gets smoothness in the normal direction at the boundary ``for free'' from the smoothness in the tangential direction plus the fact that the boundary map is CR (the analog of being conformal in \cref{cor:smooth-boundary-extension}). It is therefore natural to ask whether regularity is implied in the normal direction for the extension obtained in \cref{cor:smooth-boundary-extension}. Unfortunately, there is no analog of the Bochner-Hartogs theorem in this setting. 
From our study of the geodesic flow near infinity, however, it follows that the extension $\oX_1\to \oX_2$ of the isometry $F$ in \cref{cor:smooth-boundary-extension} is $\cC^1$ up to the boundary in the general case, and $\cC^{\infty}$ in the asymptotically hyperbolic case; see \cref{thm:boundary-extension-nor-reg} in \cref{sec:normal-reg} below.

So far the results we have presented concern the behavior of geodesics that are ``shot towards the boundary'' from the inside. But, having reparametrized and extended the cogeodesic flow so as to be (suitably) regular up to the conformal infinity, it is also possible to ``shoot geodesics in from the boundary'' by prescribing appropriate ``initial conditions'' at infinity. The nature of these ``initial conditions'' can be seen by considering geodesics in the model hyperbolic space. If we consider $2$-dimensional hyperbolic space as the right half-plane with metric $\frac{dx^2 + dy^2}{x^2}$ then the family of geodesics coming in to a given point on the boundary consists of a horizontal ray together with the family of $\text{(semi-)}$circles tangent to that ray at the boundary point; these curves may be parametrized by their second-order Taylor coefficient at the common boundary point when locally represented as the graph of a function $y(x)$ near $x=0$. An analogous parametrization exists in the asymptotically hyperbolic case (where the geodesics are all smooth up to the boundary point), but in the general case we find that the expansion picks up a singular leading term proportional to $x^2\log x$:

\begin{theorem}\label{thm:shooting-geodesics-from-boundary}
    Let $(X, g)$ be a conformally compact Riemannian manifold with $g=\rho^{-2}h$. Let $q\in \partial X$ and let $(x, y^1, \dots, y^n)$ be a Fermi coordinate chart adapted to $\partial X$ in $\oX$ with respect to the metric $h$, with $x > 0$ in $X$ and centered at $q$. Then for any vector $u \in T_q\partial X$ there is a geodesic $\gamma: (-\infty, \infty) \to X$ in $(X,g)$ tending to $q$ at infinity such that the curve $\gamma([0,\infty)) \cup \{q\}$ is given near $q$ by $y^1(x), \dots, y^n(x)$, where 
    \begin{equation}\label{eq:geodesic-asymp}
       y^{\alpha}(x) = \mathcal{O}^{\alpha} x^2 \log x + u^{\alpha} x^2 + o(x^2),  
    \end{equation}
    for some $\mathcal{O}\in T_q\partial X$ depending only on $q$ and not on $u$. Moreover, the trace $\gamma((-\infty,\infty))$ of $\gamma$ is uniquely determined by $u$ and every geodesic for $g$ that tends toward $q$ is obtained in this way.
\end{theorem}

The obstruction $\mathcal{O}\in \mathfrak{X}(\partial X)$ to higher boundary regularity of geodesics appearing in \cref{thm:shooting-geodesics-from-boundary} can be explicitly computed. In the notation of the theorem we find that, at a given point $q\in \partial X$,
    \begin{equation}\label{eq:obstruction}
        \mathcal{O}^{\alpha} = -\frac{1}{2}\frac{h^{\alpha\beta}\kappa_{\beta}}{\kappa},
    \end{equation}
where $\kappa_{\beta}=\partial_{\beta}\kappa$ and $h^{\alpha \beta}$ is the inverse of $h_{\alpha\beta}$ (the metric induced by $h$ on $\partial X$). Note that $\mathcal{O}$ depends only on $\kappa$ and the metric induced by $h$ on $\partial X$. Moreover, if $\mathcal{O}$ is interpreted as a vector field on $\partial X$ of conformal weight $-2$ in the natural way, then it only depends on $\kappa$ and the conformal structure on $\partial X$.
Note also that when $\partial X$ is connected, $(X,g)$ is asymptotically hyperbolic if and only if $\mathcal{O}\equiv 0$ on $\partial X$. Thus the obstruction to higher boundary regularity turns out to be an obstruction to asymptotic hyperbolicity; cf. \cref{thm:AH-characterization}.

\begin{remark}\label{rem:minimal}
It is interesting to compare the asymptotic expansion \cref{eq:geodesic-asymp} for geodesics in conformally compact manifolds tending to a prescribed boundary point with the known asymptotic behavior of minimal surfaces in Poincar\'e-Einstein manifolds and (more generally) asymptotically hyperbolic spaces that are ``partially even'' to sufficiently high order. In \cite{GrahamWitten1999,Marx-Kuo2025} such minimal surfaces $Y$ are considered, for general $m=\dim Y$. In this setting the minimal surface equation degenerates at the conformal infinity, where it has indicial roots $0$ and $m+1$. A consequence of the indicial roots and the (partial) evenness of the metric is that, in suitably adapted coordinates $(x,s^1,\ldots, s^{m-1}, y^1,\ldots, y^d)$, with $s=(s^1,\ldots,s^{m-1})$ giving a local parametrization of $\partial Y\subseteq \partial X$, $x$ a boundary defining function in $\oX$ and $d=n+1-m$, when $m \geq 2$ is odd one finds that near a point on $\partial Y$ the submanifold $Y$ is the graph of $(y^1(x,s),\ldots, y^{d}(x,s))$ with 
\vspace{0.2em}
\begin{equation}\label{eq:minimal-exp}
\begin{aligned}
    y^{\alpha}(s,x) = u_2^{\alpha}(s)x^2 + u_4^{\alpha}(s)x^4 + \text{(even powers)}  &+ \mathcal{O}^{\alpha}(s)x^{m+1}\log x \\
		&+ u^{\alpha}_{m+1}(s)x^{m+1} + o(x^{m+1}),
\end{aligned}
\end{equation}
where $u_k$ and $\mathcal{O}$ are smooth functions of $s$ (provided $\partial Y$ is smooth). When $m=1$ the boundary coordinates $s$ for $\partial Y$ disappear and the expansion predicted by the indicial roots calculation takes the form $y^{\alpha}(x) = \mathcal{O}^{\alpha} x^2 \log x + u^{\alpha}_2 x^2 + o(x^2)$. This is exactly what we have seen in \cref{thm:shooting-geodesics-from-boundary} (without the need for asymptotic hyperbolicity or evenness). From this point of view, the log term in \cref{eq:geodesic-asymp} is perhaps not a surprise. It is worth noting, however, that the log coefficient $\mathcal{O}$ in \cref{eq:geodesic-asymp} is of a different nature than the log coefficient in \cref{eq:minimal-exp}; the former is only nontrivial when $(X,g)$ fails to be asymptotically hyperbolic and so would not actually appear in the Poincar\'e-Einstein case, whereas the latter is a local conformal invariant of $\partial Y\subseteq \partial X$. One expects that there should be a simultaneous generalization of these two quantities for minimal surfaces in general conformally compact manifolds, but we do not pursue this here.
\end{remark}

\subsection{A Family of Examples}\label{sec:examples}

To illustrate our results in a simple setting, we consider now the family of metrics on the upper half-plane $X$ in $\mathbb{R}^2$  defined by
 $$g_{\varepsilon}  = \frac{dx^2 + dy^2}{y^2 e^{2 \varepsilon x}}.$$
Here $g_{\varepsilon}=\rho^{-2} h$ where $h=dx^2+dy^2$ and $\rho = ye^{\varepsilon x}$. Note that, since $y$ is the $h$-distance from the boundary, the roles of $x$ and $y$ are reversed relative to \cref{thm:shooting-geodesics-from-boundary}. When $\varepsilon =0$ we have the standard hyperbolic metric. For general $\varepsilon$, the metric $g_{\varepsilon}$ has limiting Gaussian curvature $-\kappa^2$, where $\kappa = e^{\varepsilon x}$, and so fails to be asymptotically hyperbolic when $\varepsilon \neq 0$. 
 In this example $\oX = \{y\geq 0\}$ is not compact, but this is of no consequence if we consider geodesics that start near $y=0$ with initial velocity directed downward and sufficiently close to the vertical (as in Figure \ref{fig:varying-geods} below). By \cref{thm:asymptotic-behavior-of-geodesic} the endpoints of such geodesics on the conformal infinity vary smoothly with the initial direction. This is obvious when $\varepsilon =0$, but was not previously known for $\varepsilon \neq 0$.

\begin{figure}[h]
\vspace{0.5em}
\begin{minipage}{1.2\textwidth}
\centering
\hspace{-34mm}
\includegraphics[height=3.8cm]{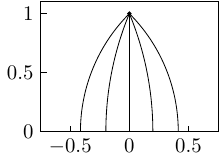} 
    \hspace{1pt}
\includegraphics[height=3.8cm]{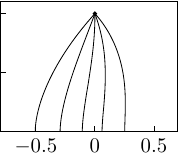}
    \hspace{1pt}
\includegraphics[height=3.8cm]{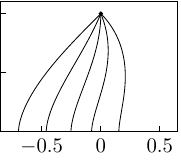}
\end{minipage}
\caption{\label{fig:varying-geods}
Five geodesics with initial position $(0,1)$ and initial velocities $(0,-1)$, $(\pm \sin(\pi/8),-\cos(\pi/8))$, and $(\pm \sin(\pi/4),-\cos(\pi/4))$ 
are plotted for the metric $g_{\varepsilon}$ for $\varepsilon=0$ (left), $\varepsilon =0.5$ (center) and $\varepsilon =1$ (right).}
\end{figure}

While the map from initial directions to boundary endpoints is smooth, the geodesics themselves are not smooth up to the boundary when $\varepsilon \neq 0$.  In this case the obstruction \cref{eq:obstruction} to boundary regularity is $\mathcal{O}= -\frac{1}{2}\varepsilon$. It follows that for a given $\varepsilon$, a geodesic reaching the boundary at a point $x_{\infty}$ is asymptotically of the form $x = x_{\infty} -\frac{1}{2}\varepsilon y^2 \log y + O(y^2)$. This captures the leading behavior of the geodesics in Figure \ref{fig:varying-geods} near $y=0$ when $\varepsilon\neq 0$.

\begin{figure}[h]
\vspace{0.5em}
\begin{minipage}{1.2\textwidth}
\centering
\hspace{-37.2mm}
\includegraphics[height=3.8cm]{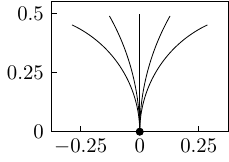} 
     \hspace{-5pt}
\includegraphics[height=3.8cm]{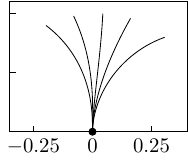}
    \hspace{-9pt}
\includegraphics[height=3.8cm]{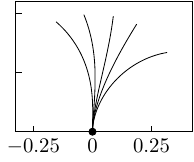}

\end{minipage}
\caption{\label{fig:shooting-geods}
The five geodesics of the metric $g_{\varepsilon}$ with prescribed asymptotic expansion $x = -\frac{1}{2}\varepsilon y^2 \log y + uy^2 + o(y^2)$, for $u=0,\pm 0.5,\pm 1$, are plotted near $(0,0)$  for $\varepsilon = 0$ (left),  $\varepsilon = 0.5$ (center) and  $\varepsilon = 1$ (right).}
\end{figure}

Figure \ref{fig:varying-geods} was produced using the extended system \cref{eqn:cogeo-tau-zeta} that was used to prove \cref{thm:asymptotic-behavior-of-geodesic,thm:shooting-geodesics-from-boundary}. A major advantage of this system is that it enables us to specify a geodesic emanating from a given point $x_{\infty}$ on the conformal infinity in terms of the vector $u \in T_{x_{\infty}}\partial X$ appearing in the asymptotic expansion \cref{eq:geodesic-asymp}; see \cref{eq:u-w-affine} for the relationship between $u$ and the tangential momentum variables used in \cref{eqn:cogeo-tau-zeta}. Figure \ref{fig:shooting-geods} gives an example of this for the metrics $g_{\varepsilon}$ when $\varepsilon = 0, 0.5, 1$ and $x_{\infty}=0$.

\subsection*{Outline of the Paper} 
In \cref{sec:cogeod-flow-near-infty} we study the cogeodesic flow of $g$ near $\partial X$ in Fermi coordinates with respect to the metric $h$ and, after successive changes of dependent and independent variables, obtain an equivalent system \cref{eqn:cogeo-tau-zeta} that extends to the boundary (with suitable regularity properties). 
In \cref{sec:proofs-of-main-thms} we make use of properties of the extended system \cref{eqn:cogeo-tau-zeta} to prove our main results. 
For completeness and ease of reference we have also included an appendix establishing the required regularity results for certain non-autonomous systems of first order ordinary differential equations. 

\subsection*{Acknowledgements} 
The authors would like to thank Rafe Mazzeo for his interest and encouragement, and in particular for asking the question answered by \cref{thm:AH-characterization}. They would also like to thank the anonymous referees for helpful comments. S.C. gratefully acknowledges support provided by the Simons Foundation, grant MPS-TSM-00002876.

\section{The cogeodesic flow near the conformal infinity}\label{sec:cogeod-flow-near-infty}

Let $(X,g)$ be a conformally compact manifold. We shall begin with a consideration of the cogeodesic flow for the metric $g$ near the conformal infinity $\partial X$. 
The cogeodesic flow is, of course, just a convenient way of writing the geodesic equations for $g$ as a first order system. By successive changes of variables, we will show that in a certain sense the (rescaled) cogeodesic flow can be extended to $\partial X$.

\subsection{Preliminary observations}\label{subsec:prelim-obs}

 As in the introduction we fix a smooth defining function $\rho$ for $\partial X$ with $\rho>0$ on $X$ and let $h$ be the metric on $\oX$ such that $g=\rho^{-2}h$. Note that the metric $g$ is complete, since $\log \rho$ gives an exhaustion function with bounded gradient. Again we let $x$ denote the $h$-distance from $\partial X$ and use the flow of $\mathrm{grad}_h \hspace{0.8pt}x$ to identify $\oU\subseteq \oX$ with $[0,\delta]\times \partial X$, $\delta>0$, so that $x$ becomes the coordinate on the first factor. We may therefore represent a point $p\in \oU$ as $(x,x')\in [0,\delta]\times \partial X$. We let $U$ be the interior of $\oU$. Fixing a local coordinate system $(x^1,\ldots, x^n)$ for $\partial X$ near a given point we obtain a Fermi coordinate system $(x^0, x^1,\ldots, x^n)$ for $\oX$ with $x^0=-x$. We denote the corresponding fiber coordinates for $T^*\oX$ by $(\xi_0, \xi_1, \dots, \xi_n)$. Note that $x^0$ and $\xi_0$ are well defined on $\oU$ (independent of the choice of local coordinates for $\partial X$). Although for convenience we make use of coordinates for $\partial X$, our discussion of the cogeodesic flow near $\partial X$ really only requires the decomposition $\oU \cong [0,\delta]\times \partial X$. When working in local coordinates the Latin indices $i, j,k, \ldots$ will have range $0,\ldots, n$ whereas the Greek indices $\alpha, \beta, \ldots$ will have range $1, \dots, n$.

In an arbitrary local coordinate system for $X$ the Hamiltonian for the cogeodesic flow of $g$ is given by
\begin{equation}
H = \frac{1}{2}g^{ij}\xi_i\xi_j = \frac{1}{2}\rho^2h^{ij}\xi_i\xi_j.
\end{equation}
The integral curves $(x^i(t), \xi_i(t))$ of the flow satisfy
\begin{equation}\label{eq:orig-cogeod-flow}
    \begin{aligned}
    \dot{x}^i &= \frac{\partial H}{\partial \xi_i} = \rho^2 h^{ij}\xi_j,\\
    \dot{\xi}_i&= - \frac{\partial H}{\partial x^i} = - \rho \rho_i h^{jk}\xi_j\xi_k - \frac{1}{2} \rho^2 \partial_ih^{jk}\xi_j\xi_k,
\end{aligned}
\end{equation}
where $\partial_i = \frac{\partial}{\partial x^i}$, and $\rho_i = \partial_i \rho$. Restricting our attention to unit-speed geodesics, we shall only consider the integral curves that further satisfy the energy surface equation
\begin{equation}\label{es}
    \rho^2 h^{ij}\xi_i\xi_j = 1.
\end{equation}
On the energy surface (\ref{es}) the cogeodesic flow takes the form
\begin{equation}\label{eq:cog-flow-ES}
    \begin{aligned}
    \dot{x}^i &= \rho^2 h^{ij}\xi_j,\\
    \dot{\xi}_i&=  - \rho^{-1} \rho_i  - \frac{1}{2} \rho^2 \partial_ih^{jk}\xi_j\xi_k.
\end{aligned}
\end{equation}
Since \cref{eq:orig-cogeod-flow} and \cref{eq:cog-flow-ES} only differ in the equation for $\dot{\xi}_i$ by a term proportional to $(1-\rho^2 h^{ij}\xi_i\xi_j)$ and since the cogeodesic flow preserves the level sets of $\rho^2 h^{ij}\xi_i\xi_j$, it follows that if a solution of \cref{eq:cog-flow-ES} satisfies \cref{es} at one time then by uniqueness (since the corresponding solution of the cogeodesic flow will also solve \cref{eq:cog-flow-ES}) it follows that \cref{es} holds for all time and $(x^i(t), \xi_i(t))$ also solves \cref{eq:orig-cogeod-flow}.

In Fermi coordinates with respect to the metric $h$ adapted to $\partial X$ it is natural to write the system \cref{eq:cog-flow-ES} as
\begin{equation}\label{eq:cog-flow-ES-2}
    \begin{aligned}
    \dot{x}^0 &= \rho^2 \xi_0,\\
    \dot{x}^{\alpha} &= \rho^2 h^{\alpha \beta}\xi_{\beta},\\
    \dot{\xi}_0&=  - \rho^{-1} \rho_0 - \frac{1}{2} \rho^2 \partial_0h^{\beta\gamma}\xi_{\beta}\xi_{\gamma}\\
    \dot{\xi}_{\alpha}&=  - \rho^{-1} \rho_{\alpha}  - \frac{1}{2} \rho^2 \partial_{\alpha}h^{\beta\gamma}\xi_{\beta}\xi_{\gamma},
\end{aligned}
\end{equation}
where we have used that $h^{00}=1$ and $h^{0\beta}=0$. Recalling that $x^0=-x$, we note that by \cref{eq:rho-vs-x} we have
\begin{equation} \label{eq:rho0}
    \rho_0 = -\kappa + O(x)
\end{equation}
and 
\begin{equation}\label{eq:rho-al}
    \rho_{\alpha} = \kappa_{\alpha}x + O(x^2),
\end{equation}
where $\kappa_{\alpha}=\partial_{\alpha}\kappa$.
In particular, while $\rho^{-1}\rho_0 =  -x^{-1} + O(1)$ blows up as $x\to 0$, $ \rho^{-1} \rho_{\alpha} = \kappa^{-1}\kappa_{\al} + O(x)$ extends smoothly to the boundary $x=0$.
Note also that the energy surface equation becomes
\begin{equation}\label{eq:ES-Fermi}
    \rho^2\xi_0^2 + \rho^2h^{\alpha\beta}\xi_{\alpha}\xi_{\beta} = 1.
\end{equation}

An example of Fermi coordinates in a neighborhood of a given boundary point for the hyperbolic ball is obtained by identifying hyperbolic space with the half space model (sending the given boundary point to the origin) and taking $h$ to be the Euclidean metric. In this case it is easy to see that along any geodesic the quantity $\xi_{\alpha}$ remains bounded as the geodesic parameter time $t\to\infty$ (since $\xi_{\alpha} = \rho^{-2}\dot{x}^{\alpha}$ with $\rho \sim e^{-t}$ and $\dot{x}^{\alpha}=O(e^{-2t})$). From the energy surface equation \cref{eq:ES-Fermi} it follows then that $\rho \xi_0$ must tend to  $\pm 1$. 

Returning to the general setting, we therefore introduce the new variable 
\begin{equation}
    \zeta_0 = \rho \xi_0.
\end{equation} 
With this change of variables,  since $\dot{\zeta}_0=\dot{\rho}\xi_0+\rho\dot{\xi}_0$ with $\dot{\rho}=\rho_0\dot{x}^0+\rho_{\alpha}\dot{x}^{\alpha}$, the system \cref{eq:cog-flow-ES-2} becomes
\begin{equation}\label{eqn:cogeo:initial}
    \begin{aligned}
    \dot{x}^0 &= \rho\zeta_0,\\
    \dot{x}^\alpha &= \rho^2 h^{\alpha \beta}\xi_\beta,\\
    \dot{\zeta}_0 &=
    -\rho_0(1-\zeta^2_0)+\rho \rho^{\beta}\xi_\beta\zeta_0  - \frac{1}{2} \rho^3 \partial_0h^{\beta\gamma}\xi_{\beta}\xi_{\gamma}, \\ 
    \dot{\xi}_\alpha&=  - \rho^{-1} \rho_\alpha  - \frac{1}{2} \rho^2 \partial_{\alpha}h^{\beta\gamma}\xi_{\beta}\xi_{\gamma},
\end{aligned}
\end{equation}
where $\rho^{\beta}=h^{i\beta}\rho_i=h^{\alpha\beta}\rho_{\alpha}$. 
The energy surface equation then becomes 
\begin{equation}\label{eq:ES-zeta}
    \zeta_0^2 + \rho^2h^{\alpha\beta}\xi_{\alpha}\xi_{\beta} = 1.
\end{equation}
Note, in particular, that we must therefore have $|\zeta_0|\leq 1$ and $|\xi_{\alpha}|=O(\rho^{-1})$ along the flow.  Our observations from the hyperbolic case are generalized in the following proposition:

\begin{proposition}\label{prop:elem-obs}
    After shrinking $U$ (equivalently $\delta>0$) if necessary, the following hold for all geodesics $\gamma:[0,\infty)\to X$ with $\gamma(0)\in U$ and $\dot{x}^0(0)>0\!:$
    \begin{itemize}
        \item[(i)] $\gamma$ remains in $U$, $\dot{x}^0(t)>0$ for all $t$ and $x^0\to 0$ as $t\to\infty;$
        \item[(ii)]  $\zeta_0 \to 1$ as $t\to\infty$.
    \end{itemize}
\end{proposition}

\begin{proof}
Fix such a geodesic $\gamma$. Note that $\dot{x}^0(t)>0$ if and only if $\zeta_0(t)>0$. Note also that, by \cref{eq:rho0}, after shrinking $\delta>0$ if necessary, we may assume that $-\rho_0 >0$ in $\oU$. To prove the first part of (i) we observe that if $t\geq 0$ is such that $\gamma(t)\in U$ then either $|\zeta_0(t)|\geq 1/2$ or we have $|\zeta_0(t)| < 1/2$ and hence (in an appropriate Fermi coordinate chart)
\begin{equation}\label{eq:zeta-dot-ineq}
    \dot\zeta_0 > -\rho_0\left(1-\frac{1}{4}\right) +\rho \rho^{\beta}\xi_\beta\zeta_0  - \frac{1}{2} \rho^3 \partial_0h^{\beta\gamma}\xi_{\beta}\xi_{\gamma} =  \frac{3}{4}\kappa + O(\rho).
\end{equation}
Although \cref{eq:zeta-dot-ineq} depends on a choice of Fermi coordinate chart containing $\gamma(t)\in U$, since $\kappa>0$ on $\partial X$ and $\partial X$ is compact, we may cover $U$ with a finite number of preferred coordinate charts and take $\delta>0$ small enough such that the right-hand side $\frac{3}{4}\kappa + O(\rho)$ of \cref{eq:zeta-dot-ineq} is always positive in such a chart. Thus, if $\zeta_0$ is initially positive then either $\zeta_0\geq 1/2$ or $\dot\zeta_0>0$ at any time such that $\gamma(t)\in U$. Clearly, then, $\dot{x}^0=\rho\zeta_0$ remains positive, so that $x^0(t)$ is strictly increasing and $\gamma (t)$ can never leave $U$. 

To see that $x^0\to 0$ as $t\to \infty$ we note that $\lim_{t\to \infty}x^0(t)$ exists (since $x^0(t)$ is increasing and bounded above). If this limit were not zero then there would be $\epsilon>0$ such that $\rho(\gamma(t))\geq \epsilon$ for all $t$, so that the positive quantity $\dot{x}^0=\rho\zeta_0$ would be bounded away from zero along $\gamma$ (since $\zeta_0(t)\geq \mathrm{min}\{\zeta_0(0), 1/2\}>0$). But this would imply that $\lim_{t\to\infty}x^0(t)=+\infty$, which is impossible. It follows that we must have $\lim_{t\to \infty}x^0(t)=0$.

It remains to prove (ii). Suppose $\zeta_0(0)>0$, so that $\zeta_0(t)>0$ for all $t$.  Arguing similarly to the above, for each integer $m\geq 2$ there is a $\delta_m >0$ and a corresponding subset $U_m\cong (0,\delta_m)\times\partial X$ such that when $\gamma(t)\in U_m$ then either $\zeta_0(t)\geq 1-1/m$ or we have $\zeta_0(t) < 1-1/m$ and hence (in one of our preferred coordinate charts)
\begin{equation}\label{eq:zeta-dot-ineq-n}
\begin{aligned}
    \dot\zeta_0 > -\rho_0\left(1-(1-1/m)^2\right) +\rho \rho^{\beta}\xi_\beta\zeta_0  - \frac{1}{2} \rho^3 \partial_0h^{\beta\gamma}\xi_{\beta}\xi_{\gamma} &=  \frac{2m-1}{m^2}\kappa + O(\rho) \\ &> c_m>0.
\end{aligned}
\end{equation}
Since $\lim_{t\to \infty}x^0(t)=0$, the geodesic $\gamma$ eventually enters every set $U_m$. Once $\gamma$ enters a given set $U_m$, $\zeta_0$ must eventually become larger than $1-1/m$ (since otherwise $\dot{\zeta}_0>c_m>0$ for all $t$ near $\infty$ and hence $\lim_{t\to\infty}\zeta_0(t)=+\infty$, which is impossible) and afterward $\zeta_0$ can never again decrease below $1-1/m$. It follows that $\lim_{t\to\infty}\zeta_0(t)=1$.
\end{proof}

We henceforth assume that $U$ (equivalently, $\delta>0$) is taken small enough that \cref{prop:elem-obs} holds. For convenience, we'll also suppose that $\delta<1$ so that $\log x<0$ in $U$.

It follows that if $\gamma:[0,\infty)\to X$  is a geodesic for $g$ with $\gamma(0)\in U$ and $\dot{x}^0(0)>0$ then $x^0(t)$ is a monotone function of $t$ and so can be used as a parameter for $\gamma$. When viewing $x^0$ as a parameter, we will write it as $\tau$. Occasionally we'll write $\tilde{\gamma}:[-\delta_0,0)$ for the reparametrized geodesic, so that $\tilde{\gamma}(\tau) = \gamma(t)$ where $\tau=\tau(t)$ (or $t=t(\tau)$). By the previous lemma, $\tau\to 0$ as $t\to\infty$, but it will be useful to have a more precise asymptotic relationship. For now, the following lemma will suffice. 
\begin{lemma}\label{lem:t-tau-basic-asym}
    Let $\gamma:[0,\infty)\to X$ be a geodesic with $\gamma(0)\in U$ and $\dot{x}^0(0)>0$. Then there are constants $\kappa_1,\kappa_2>0$ such that $x(0)e^{-\kappa_1 t} \leq x(t) \leq x(0)e^{-\kappa_2 t}$ for all $t$. Hence there are constants $C_1,C_2>0$ such that 
\begin{equation}
    -C_1 \log|\tau| \leq t \leq -C_2 \log |\tau|
\end{equation}
for all $t$.
\end{lemma}
\begin{proof}
Since $\zeta_0>0$ for all $t$ and $\zeta_0\to 1$ as $t\to\infty$, $\zeta_0$ is bounded above and below by positive constants. Similarly, $\rho/x$ is a bounded positive function on $U$ that is bounded away from $0$. Hence, the equation $\dot{x}^0 = \rho\zeta_0$ implies that $-\kappa_1 x\leq \dot{x} \leq -\kappa_2 x$ for some positive constants $\kappa_1,\kappa_2$. Thus $-\kappa_1 \leq \frac{d}{dt}\log x \leq -\kappa_2$ and hence $x(0)e^{-\kappa_1 t} \leq x(t)\leq x(0) e^{-\kappa_2 t}$ for all $t$. This proves the first claim. The second follows easily from the first.
\end{proof}

\begin{proposition}\label{prop:def-bdry-pt}
    For all geodesics $\gamma:[0,\infty)\to X$ with $\gamma(0)\in U$, if $\dot{x}^0(0)>0$ then $\gamma(t)$ tends to a point $\gamma_{\infty}\in \partial X$ as $t\to\infty$.
\end{proposition}

\begin{proof}
Suppose $\gamma$ is such a geodesic, with $\dot{x}^0(0)>0$. Since $\gamma$ remains in $U\cong (0,\delta)\times \partial X$, $\gamma$ projects on the second factor to a well-defined smooth curve $\gamma^{\partial X}:[0,\infty)\to \partial X$. In a Fermi coordinate chart where $\dot{\gamma}$ has components $(\dot{x}^0, \dot{x}^{\alpha})$, $\dot{\gamma}^{\partial X}$ has components $\dot{x}^{\alpha}$. The energy surface equation \cref{eq:ES-zeta} combined with $\dot{x}^\alpha = \rho^2 h^{\alpha \beta}\xi_\beta$ gives us the coordinate-independent identity, 
\begin{equation}
    \zeta_0^2 + \rho^{-2}|\dot{\gamma}^{\partial X}|_{h_x}^2 =1,
\end{equation}
where $h_x$ is the metric induced by $h$ on the hypersurface $\{x\}\times \partial X$ in  $\oU\cong[0,\delta]\times\partial X$, viewed as a metric on $\partial X$.
Since $\zeta_0\to 1$, it follows that $\rho(\gamma(t))^{-2}|\dot{\gamma}^{\partial X}(t)|_{h_{x(t)}}^2 \to 0$ as $t\to \infty$. Since $h$ is a smooth metric on $\oX$, $h_{x(t)}$ and $h_{0}$ are quasi-isometric, uniformly in $t$. It follows that $\rho(\gamma(t))^{-2}|\dot{\gamma}^{\partial X}(t)|_{h_{0}}^2$ goes to $0$ as $t\to \infty$, i.e. $|\dot{\gamma}^{\partial X}(t)|_{h_{0}}=o(\rho)$. In particular, $|\dot{\gamma}^{\partial X}(t)|_{h_{0}}=O(x)$, since this is equivalent to $|\dot{\gamma}^{\partial X}(t)|_{h_{0}}=O(\rho)$. Now, by \cref{lem:t-tau-basic-asym}, there is a constant $\kappa_2>0$ such that $x(t)\leq x(0)e^{-\kappa_2 t}$. Hence $|\dot{\gamma}^{\partial X}(t)|_{h_{0}}=O(e^{-\kappa_2 t})$. Since $(\partial X, h_0)$ is a compact Riemannian manifold (hence complete) it follows that $\gamma^{\partial X}(t)$ must tend to some fixed point $\gamma_{\infty}\in \partial X$ as $t\to \infty$. Since $\lim_{t\to\infty}x^{0}(t)=0$ it then follows that $\gamma(t)$ must also tend to $\gamma_{\infty}\in \partial X$.
\end{proof}
    
In the above argument, we made a point of not relying on any particular local coordinate system. If we had known that $\gamma$ remained in a single Fermi coordinate chart, then we could have argued in essentially the same way to show that $\dot{x}^{\alpha}(t)$ was $o(e^{-Ct})$ and hence integrable, so that $\lim_{t\to\infty} x^{\alpha}(t)$ must exist. Of course, now that we know $\gamma$ tends to a specific boundary point $\gamma_{\infty}$, we know that $\gamma(t)$ eventually remains in some Fermi coordinate chart. Since the questions we are interested in may all be studied locally near the boundary endpoint $\gamma_{\infty}$ of a given geodesic $\gamma:[0,\infty)\to X$ that leaves all compact sets, in the following we will work in a fixed Fermi coordinate chart. 
\begin{remark}\label{rem:single-Fermi-chart}
Note that from the proof 
above we know that, in such a chart, $\rho^{-1}\dot{x}^{\alpha}\to 0$ (equivalently, $\rho \xi_{\alpha}\to 0$) as $t\to \infty$. This shows us that, although it is natural to replace the variable $\xi_0$ by $\zeta_0=\rho\xi_0$, if we wish to be able to ``shoot geodesics in'' from $\partial X$ then we should not replace the variables $\xi_{\alpha}$ by $\zeta_{\alpha}=\rho\xi_{\alpha}$ (the $\zeta_{\alpha}$ go to zero along every geodesic, so their limiting values at $t=\infty$ clearly fail to distinguish between the different geodesics that tend to a given point $q\in \partial X$; hence the limiting values of these variables cannot be chosen as initial data for a ``flow'' that goes in the opposite direction). In hyperbolic space $(\mathbb{B}^{n+1},g)$ it is easy to see that the family of geodesics tending towards a given boundary point is parametrized by the limiting values of the variables $\xi_{\alpha}$ (not of the $\zeta_{\alpha}$). Hence, in that case, the $\xi_{\alpha}$ are already ``good'' variables. This continues to be the case for asymptotically hyperbolic manifolds, but we will see below that we need to make a further change of variables to handle general conformally compact manifolds. 
\end{remark}

\subsection{Extending to the boundary}\label{sec:extending}
In light of \cref{prop:elem-obs}, for geodesics $\gamma$ in $U\subset X$ with $\dot{x}^0(0)>0$ it is natural to replace the parameter time variable $t$ with the boundary-defining function $\tau = x^0$ and attempt to extend the (reparametrized) cogeodesic flow up to $\partial X$. Before we do this it is helpful to use the energy surface equation to rewrite the $-\rho_0(1-\zeta_0^2)$ term in \cref{eqn:cogeo:initial} as $-\rho^2 \rho_0h^{\be\ga}\xi_{\be}\xi_{\ga}$. We obtain the system:
\begin{equation}\label{eqn:cogeo:initial-2}
    \begin{aligned}
    \dot{x}^0 &= \rho\zeta_0,\\
    \dot{x}^\alpha &= \rho^2 h^{\alpha \beta}\xi_\beta,\\
    \dot{\zeta}_0 &=
    \rho \rho^{\beta}\xi_\beta\zeta_0 -\rho^2 \rho_0h^{\be\ga}\xi_{\be}\xi_{\ga} - \frac{1}{2} \rho^3 \partial_0h^{\beta\gamma}\xi_{\beta}\xi_{\gamma}, \\ 
    \dot{\xi}_\alpha&=  - \rho^{-1} \rho_\alpha  - \frac{1}{2} \rho^2 \partial_{\alpha}h^{\beta\gamma}\xi_{\beta}\xi_{\gamma}.
\end{aligned}
\end{equation}

Before we make our change of dependent variable, we make one simple observation concerning solutions of \cref{eqn:cogeo:initial-2}:
\begin{lemma} \label{lem:v-al-log-rho-bound}
Let $\gamma:[0,\infty)\to X$ be a geodesic for $g$ that remains in a given Fermi coordinate system for all time. Then, for the corresponding solution of \cref{eqn:cogeo:initial-2} we have $\xi_{\al} = O(t)$ for $t$ large, and hence $\xi_{\al} = O(\log \rho)$ for $\rho$ near $0$. 
\end{lemma}
\begin{proof}
Note that, by the energy identity, $\xi_{\al} = O(\rho^{-1})$ and $\zeta_0$ is bounded. (In fact the energy estimate tells us that $\xi_{\al} = o(\rho^{-1})$, since $\zeta_0\to 1$ as $t\to\infty$, but we do not need this for the argument and we get a stronger result as our conclusion.) 
Since $\rho^{-1}\rho_{\al} = \kappa^{-1}\kappa_{\al} + O(\rho)$, where $\kappa_{\al} = \partial_{\al}\kappa$, it follows from the equation for $\dot{\xi}_{\al}$ in \cref{eqn:cogeo:initial-2} that $\dot{\xi}_{\al}$ is uniformly bounded in $t$ and hence that $\xi_{\al} = O(t)$ for $t$ large. By \cref{lem:t-tau-basic-asym}, this implies $\xi_{\al} = O(\log|\tau|)$ for $\tau$ near $0$, equivalently, $\xi_{\al} = O(\log \rho)$ for $\rho$ near $0$. This proves the lemma.
\end{proof}
\begin{remark}
    It follows from the discussion below that in the asymptotically hyperbolic case $\xi_{\al} = O(1)$, but that in the general conformally compact case the bound $\xi_{\al} = O(\log \rho)$ cannot be improved. Below we will replace $\xi_{\alpha}$ by a variable $\zeta_{\al}$ that is always $O(1)$ and tends to a finite limit.
\end{remark}

Restricting our attention to geodesics near the conformal infinity with $\dot{x}^0(0)>0$, we now replace the parameter time $t$ with the new independent variable $\tau=x^0 \in [-\delta,0)$. From \cref{eqn:cogeo:initial-2}  we obtain the system
\begin{equation}\label{eqn:cogeo-tau}
    \begin{aligned}
    \frac{dx^{\alpha}}{d\tau} &= \rho \frac{h^{\al\be}\xi_{\be}}{\zeta_0},\\
    \frac{d\zeta_0}{d\tau} &=
     \rho^{\beta}\xi_{\beta} - \rho\frac{\rho_0h^{\be \ga}\xi_{\be}\xi_{\ga}}{\zeta_0} - \frac{1}{2} \rho^2 \frac{\partial_0 h^{\be \ga}\xi_{\be}\xi_{\gamma}}{\zeta_0}, \\ 
    \frac{d\xi_{\alpha}}{d\tau}&=  - \rho^{-2} \frac{\rho_{\al}}{\zeta_0} -\frac{1}{2}\rho \frac{\partial_{\al}h^{\be \ga} \xi_{\be}\xi_{\ga}}{\zeta_0}.
\end{aligned}
\end{equation}
Recall that $\zeta_0$ is always positive for such geodesics. Hence the only potential singularity at $\tau=0$ in the above system comes from the $\rho^{-2}\rho_{\alpha}$ coefficient in the third equation. Note that from \cref{eq:rho-vs-x} we have
$$
\rho^{-2}\rho_{\al} = x^{-1}\kappa^{-2}\kappa_{\al} + O(1),
$$
where the $O(1)$ term is smooth up to $\partial X$. It follows that the system \cref{eqn:cogeo-tau} is regular at $\tau=0$ near a given boundary point if and only if $(X,g)$ is asymptotically hyperbolic in a neighborhood of that boundary point.

In order to remove the singularity in the general case, we make another change of variables.
Let 
\begin{equation}\label{eq:A}
    A_{\alpha}(x^0, x^1, \dots, x^n,\zeta_0) = \frac{1}{\zeta_0} \cdot \frac{\kappa_{\alpha}}{\kappa^2} \cdot\log |x^0|
\end{equation}
and define new variables $\zeta_{\al}$ by 
\begin{equation}\label{def-eq:zeta-al}
\zeta_{\al} = \xi_{\al}-A_{\al}.
\end{equation}

Before we write the system in terms of these new variables it will be useful to introduce some further notation. Let $E_{\alpha}$ denote the remainder term in the expression $-\rho^{-2}\rho_{\al} = (x^0)^{-1}\kappa^{-2}\kappa_{\al} + E_{\al}$, which is smooth up to $\partial X$. Then 
\begin{equation}
-\frac{\rho^{-2}\rho_{\al}}{\zeta_0} = \frac{\partial A_{\al}}{\partial x^0} + \frac{E_{\al}}{\zeta_0}.
\end{equation}
We also note that by \cref{eq:rho-vs-x} we can write $\rho^{\beta}$ as $\rho k^{\beta}$, where $k^{\beta}$ is smooth up to $\partial X$.
In terms of the new variables $(x^1,\ldots, x^n, \zeta_0, \zeta_1,\ldots \zeta_n)$, the system \cref{eqn:cogeo-tau} then becomes
\begin{equation}\label{eqn:cogeo-tau-zeta}
    \begin{aligned}
    \frac{dx^{\alpha}}{d\tau} &=Y^{\al} := \rho \frac{h^{\al\be}(\zeta_{\be}+A_{\be})}{\zeta_0},\\
    \frac{d\zeta_0}{d\tau} &= Z_{0} := \rho k^{\beta}(\zeta_{\beta}+A_{\be}) - \rho\frac{\rho_0h^{\be \ga}(\zeta_{\be}+A_{\be})(\zeta_{\ga}+A_{\ga})}{\zeta_0} \\
		& \qquad \qquad \qquad \qquad \qquad \qquad \quad - \frac{1}{2} \rho^2 \frac{\partial_0 h^{\be \ga}(\zeta_{\be}+A_{\be})(\zeta_{\ga}+A_{\ga})}{\zeta_0}, \\ 
    \frac{d\zeta_{\alpha}}{d\tau}&= Z_{\al} := \frac{E_{\al}}{\zeta_0} - \frac{1}{2}\rho\frac{\partial_{\al}h^{\be\ga}(\zeta_{\be}+A_{\be})(\zeta_{\ga}+A_{\ga})}{\zeta_0} -Y^{\beta}\partial_{\beta}A_{\al} - Z_0 \partial_{\zeta_0} A_{\al},
\end{aligned}
\end{equation}
which now makes sense for $\tau \in [-\delta,0]$ (since $x\log x$, $x (\log x)^2$ and $x (\log x)^3$ extend continuously from $x>0$ to $x=0$).

Letting $y$ stand for the $(2n+1)$-tuple of dependent variables $(x^1,\ldots, x^n,$ $\zeta_0, \zeta_1,\ldots \zeta_n)$, the system \cref{eqn:cogeo-tau-zeta}
is of the form
\begin{equation}\label{eq:ODE}
    \frac{dy}{d\tau} = V(\tau,y),
\end{equation}
where $V$ is $\cC^{\infty}$ in $(\tau,y)$ for $\tau<0$ but only continuous at $\tau =0$. Note, however, that we have a good deal more control over the right-hand side $V(\tau, y)$ than this would suggest. Indeed, from \cref{eqn:cogeo-tau-zeta} and \cref{eq:A} it follows that $V(\tau,y)$ is of the form
\begin{equation}\label{eqn:cogeo-tau-w-RHS-log-expansion}
\begin{aligned}
    V(\tau,y) &= V_0(\tau,y) + \tau(\log|\tau|) V_1(\tau,y) + \tau(\log|\tau|)^2V_2(\tau,y) \\ & \qquad\qquad \qquad\qquad\qquad \qquad\qquad \quad\;\;+ \tau(\log|\tau|)^3V_3(\tau,y) ,
\end{aligned}
\end{equation}
where the functions $V_0, V_1, V_2, V_3$ are smooth up to $\tau=0$ in $(\tau,y)$. In particular, this means that not only $V$ but also all of its $y$-partial derivatives are  $\cC^{\infty}$ in $(\tau,y)$ for $\tau<0$ and continuous at $\tau =0$. This allows us to make use of the regularity results established in \cref{app:reg} to prove our main results.

\begin{remark}
The above observation puts us in a setting very similar to that of Fefferman in his work on the boundary behavior of geodesics for the Bergman metric. In \cite{Fefferman1974} Fefferman obtained an asymptotic expansion for the Bergman kernel on the diagonal for a smoothly bounded strongly pseudoconvex domain in $\mathbb{C}^n$, from which it follows that for geodesics tending toward the boundary the geodesic equations can be written in the form $\frac{dy}{dT} = F$, where $T$ is the distance from the boundary and $F$ is a smooth function of $y, T$ and $T^{n-1}\log T$. A key point of \cite{Fefferman1974} is that the presence of the log terms is ultimately harmless when it comes to the question of smoothness of the dependence of the boundary endpoint on the initial direction of the geodesic. The same is true for us here. What is different here, though, is that the log terms are arising not from log terms in an expansion for the metric (though one can allow for such) but rather from a singularity in the geodesic equations that only arises when asymptotic hyperbolicity fails.
\end{remark}

\subsection{Extending solutions to the boundary}

Before we prove our main results, however, we wish to establish that geodesics for $g$ that approach $\partial X$ give rise to solutions of \cref{eqn:cogeo-tau-zeta} that extend to $\tau=0$. This is shown in the following lemma.

\begin{lemma}\label{lem:ext-to-tau-eq-0}
Let $\gamma:[0,\infty)\to X$ be a geodesic for $g$ with $\gamma(0)\in U$ and $\dot{x}^0(0)>0$ and let  $(x^0, x^1, \ldots, x^n)$ be a Fermi coordinate system in a neighborhood of $\gamma_{\infty}$.  Then the smooth solution $(x^{\alpha}, \zeta_0, \zeta_{\alpha})$ of \cref{eqn:cogeo-tau-zeta} obtained from $\gamma$ for $\tau \in [-\delta_0,0)$ extends to $\tau=0$, giving a $\cC^1$ solution for $\tau\in [-\delta_0,0]$.
\end{lemma}
\begin{proof}
As observed previously, the system \cref{eqn:cogeo-tau-zeta} is of the form \cref{eq:ODE} with $V(\tau,y)$ smooth in $y$ for each fixed $\tau\in [-\delta,0]$. Moreover, the $y$-partial derivatives of $V(\tau,y)$ vary continuously with $\tau\in [-\delta,0]$, and hence $V(\tau,y)$ is locally Lipschitz in $y$ uniformly in $\tau\in [-\delta,0]$.  It therefore follows from the standard Picard-Lindel\"of existence theorem that if $y(\tau)$ solves \cref{eqn:cogeo-tau-zeta} for $\tau\in [-\delta_0,0)$ and remains in a bounded set (with $\zeta_0(\tau)$ bounded away from $0$ and the $x^{\al}(\tau)$ bounded away from the boundary of the region where the coordinate system is defined) for all $\tau\in [-\delta_0,0)$ then the solution $y(\tau)$ extends to $\tau\in [-\delta_0,0]$ and is $\cC^1$ in $\tau$. Since the functions $x^{\alpha}$ tend to the coordinates of $\gamma_{\infty}$ as $\tau\to 0^-$ and  $\zeta_0 \to 1$ as $\tau\to 0^-$, it remains only to see that the functions $\zeta_{\alpha}$ are also bounded for $\tau\in [-\delta_0,0)$. To see this it is helpful to write the system \cref{eqn:cogeo-tau-zeta} back in terms of the variables $\xi_{\al} = \zeta_{\al}+A_{\al}$. We obtain (cf.\ \cref{eqn:cogeo-tau}):
\begin{equation}\label{eqn:cogeo-tau-zeta-with-xi}
    \begin{aligned}
    \frac{dx^{\alpha}}{d\tau} &=Y^{\al} := \rho \frac{h^{\al\be}\xi_{\be}}{\zeta_0},\\
    \frac{d\zeta_0}{d\tau} &= Z_{0} := \rho k^{\beta}\xi_{\beta} - \rho\frac{\rho_0h^{\be \ga}\xi_{\be}\xi_{\ga}}{\zeta_0} - \frac{1}{2} \rho^2 \frac{\partial_0 h^{\be \ga}\xi_{\be}\xi_{\ga}}{\zeta_0}, \\ 
    \frac{d\zeta_{\alpha}}{d\tau}&= Z_{\al} := \frac{E_{\al}}{\zeta_0} - \frac{1}{2}\rho\frac{\partial_{\al}h^{\be\ga}\xi_{\be}\xi_{\ga}}{\zeta_0} -Y^{\beta}\partial_{\beta}A_{\al} - Z_0 \partial_{\zeta_0} A_{\al},
\end{aligned}
\end{equation}
Recalling from \cref{lem:v-al-log-rho-bound} that $\xi_{\al}=O(\log \rho)$, it follows that $Y^{\al} = O(\rho \log \rho)$ and $Z_0=O(\rho(\log\rho)^2)$. Hence, all terms in the expression for $Z_{\al}$ remain bounded; indeed, the first term, $E_{\al}/\zeta_0$, is smooth up to $\partial X$ and therefore bounded; noting that $Y^{\beta}\partial_{\beta}A_{\al} =O(\rho(\log\rho)^2)$ and $Z_0 \partial_{\zeta_0} A_{\al} = O(\rho(\log\rho)^3)$, it is easy to see that the other terms go to zero as $\tau\to 0^-$. It follows easily that $\lim_{\tau \to 0^-} \zeta_{\al}(\tau)$ exists, and, in particular, that the functions $\zeta_{\al}$ are bounded for $\tau\in [-\delta_0,0)$. Hence the solution $(x^{\alpha}, \zeta_0, \zeta_{\alpha})$ extends to $\tau =0$ as a $\cC^1$ solution of \cref{eqn:cogeo-tau-zeta}. This proves the lemma.
\end{proof}

\section{Proofs of the main results}\label{sec:proofs-of-main-thms}

\subsection{Proof of Theorem \ref{thm:asymptotic-behavior-of-geodesic}}

Here we prove \cref{thm:asymptotic-behavior-of-geodesic} from the introduction, which shows, in particular, that geodesics leaving all compact sets tend to a definite point in the conformal infinity and that this point varies smoothly as one varies the initial direction of the geodesic.  

Let $(X, g)$ be a conformally compact Riemannian manifold with $g= \rho^{-2}h$. Let $\gamma: [0, \infty) \to X$ be a geodesic for $g$ that leaves all compact sets. 
To prove (i) we note that, since $\gamma$ leaves all compact sets, $\gamma$ must eventually enter $U=\{\rho<\delta\}\subset X$. Clearly we must also have $\lim_{t\to\infty}x^0(t)=0$, where $x^0(t)=x^0(\gamma(t))$. It follows that there is a parameter time $t_0$ such that $\gamma(t_0)\in U$ and $\dot{x}^0(t_0)>0$. After an affine change of parameter, we may assume without loss of generality that $t_0=0$. It then follows from \cref{prop:def-bdry-pt} that $\gamma(t)$ tends to a definite point $\gamma_{\infty}\in \partial X$ as $t\to\infty$. This proves part (i).

To prove part (ii) we first note that we may again assume, without loss of generality, that $\gamma(0)\in U$ and that $\dot{x}^0(0)>0$. Then $\gamma$ remains in $U$ and $\dot{x}^0$ remains positive for all $t>0$. Thus we may reparametrize $\gamma$ in terms of $\tau=x^0$. Fixing a Fermi coordinate system $(x^0,x^1,\ldots,x^n)$ in a neighborhood of $\gamma_{\infty}$ we therefore obtain a (smooth) solution $(x^{\al},\zeta_0,\zeta_{\al})$ of the system \cref{eqn:cogeo-tau-zeta} for $\tau\in [-\delta_0,0)$ from $\gamma$, for some $\delta_0\in (0,\delta]$. By \cref{lem:ext-to-tau-eq-0} this solution extends to a $\cC^1$ solution of  \cref{eqn:cogeo-tau-zeta} for $\tau\in [-\delta_0,0]$. In particular, $\zeta_0$ is $\cC^1$ up to $\tau=0$. Thus, since $\zeta_0\to 1$ as $\tau\to 0^-$ it follows that $\zeta_0 = 1 + O(\tau)$. Hence $\zeta_0 = 1 + O(\rho)$, and recalling that $\dot{x}^0 = \rho \zeta_0$ we see that 
\begin{equation}
    \dot{x}^0 = \rho + O(\rho^2).
\end{equation}
Similarly, from \cref{lem:v-al-log-rho-bound} we have that $\xi_{\al} = O(\log \rho)$ and recalling that $\dot{x}^{\alpha}= \rho^2 h^{\alpha\be}\xi_{\beta}$ it follows that
\begin{equation}\label{eq:rough-dot-x-al-asymp}
    \dot{x}^{\alpha} =  O(\rho^2\log\rho).
\end{equation}
This proves (ii).
\begin{remark}\label{rem:x-alpha-dot-more-precise-asymp}
In fact, from the formula $\dot{x}^{\alpha} = \rho^2 h^{\al\be}\xi_{\be}=\rho^2 h^{\al\be}(\zeta_{\be}+A_{\beta})$ and the fact that $\zeta_{\be}$ is bounded we obtain a more precise asymptotic statement than \cref{eq:rough-dot-x-al-asymp}; see \cref{lem:asymp-xalphadot} below. 
\end{remark}

We now turn to the proof of (iii). Since the image of a geodesic is a smooth immersed curve, the problem localizes near $\gamma_{\infty}$. We therefore assume, without loss of generality, that $\gamma(0)\in U$, $\dot{x}^0(0)>0$ and that  $\gamma([0,\infty))\cup\{\gamma_{\infty}\}$ is contained in a single Fermi coordinate chart. As before, we reparametrize $\gamma$ using $\tau = x^0$ and obtain a smooth solution of \cref{eqn:cogeo-tau-zeta} for $\tau \in [-\delta_0,0)$ that extends to a $\cC^1$ solution for $\tau \in [-\delta_0,0]$. It follows that $\gamma([0,\infty))\cup\{\gamma_{\infty}\}$ is a $\cC^{1}$-embedded curve in $\oX$ (recall that we have localized near $\gamma_{\infty}$). 
That the curve meets the boundary orthogonally then follows from part (ii) and the fact that $(x^0,x^{\alpha})$ are Fermi coordinates for $h$, since (ii) implies that $\dot{x}^{\alpha}/\dot{x}^0 = O(\rho\log\rho)$ and hence $\dot{x}^{\alpha}/\dot{x}^0\to 0$ as $t\to\infty$, equivalently, $\frac{dx^{\alpha}}{d\tau}\to 0$ as $\tau\to 0^{-}$. Since $\gamma([0,\infty)) \cup \{\gamma_{\infty}\}$ is smooth away from $\gamma_{\infty}$ it remains only to show that the curve is $\cC^{1, \alpha}$ up to the conformal infinity, for $\alpha\in (0,1)$. To see this we note that $\frac{dx^{\alpha}}{d\tau}=\rho h^{\al\be}(\zeta_{\beta}+A_{\beta})/\zeta_0$ ($\alpha =1,\ldots , n$). Since $(x^{\alpha},\zeta_0, \zeta_{\al})$ is $\cC^1$ on $[-\delta_0,0]$ it follows that $\frac{dx^{\alpha}}{d\tau}$ is $\cC^{\alpha}$ as a function of $\tau \in [-\delta_0,0]$. Thus $\gamma([0,\infty))\cup\{\gamma_{\infty}\}$ is $\cC^{1,\alpha}$-embedded near $\partial X$. This proves (iii).

To prove (iv) we note first that for any smooth $1$-parameter family of initial conditions $\dot{\gamma}_s(0)$ in $TX$, there is a corresponding smooth $1$-parameter family of geodesics $\gamma_s:[0,\infty)\to X$ with initial velocity given by $\dot{\gamma}_s(0)$ for each $s$. Let $\gamma_s$ be such a family with $\gamma_0=\gamma$.
Let $T>0$ be such that $\gamma(T)\in U$ and $\dot{x}^0(T)>0$. Then for all $s$ near $0$ we also have $\gamma_s(T)\in U$ and $\dot{x}^0_s(T)>0$. It follows that there is a Fermi coordinate chart containing $\gamma_{\infty}$ such that the closure of $\gamma_s([T,\infty))$ lies in the chart for all $s$ near $0$. Hence there is a $\delta_0>0$ such that for all $s$ near $0$ the curve $\gamma_s$ gives rise to a solution $(x^{\alpha}_s, \zeta_0^s, \zeta_{\alpha}^s)$ of the system \cref{eqn:cogeo-tau-zeta} for $\tau \in [-\delta_0,0)$. Moreover, by \cref{lem:ext-to-tau-eq-0}, for all $s$ near $0$ the solution $(x^{\alpha}_s, \zeta_0^s, \zeta_{\alpha}^s)$ extends to $\tau =0$, giving a $\cC^1$ solution for $\tau \in [-\delta_0,0]$. Now, from the smooth dependence of $\gamma_s$ on $s$ it follows that the initial conditions $(x^{\alpha}_s(-\delta_0), \zeta_0^s(-\delta_0), \zeta_{\alpha}^s(-\delta_0))$ of these solutions vary smoothly with $s$. By the results of \cref{app:reg} it follows that the values of the solution $(x^{\alpha}_s, \zeta_0^s, \zeta_{\alpha}^s)$ at $\tau=0$ depend smoothly on the initial condition (and hence vary smoothly with $s$). In particular, the coordinates $\left.x^{\alpha}_s\right|_{\tau=0}$ of the boundary limit point $\lim_{t\to\infty}\gamma_s(t)$ of $\gamma_s$ vary smoothly with $s$. This proves (iv).

This concludes the proof of \cref{thm:asymptotic-behavior-of-geodesic}.

\subsection{Geodesics with prescribed data at the conformal infinity}\label{sec:geod-from-infty}
Before we prove our other main theorems, we need some results on going from solutions of \cref{eqn:cogeo-tau-zeta} for $\tau\in [-\delta_0,0]$ to geodesics for $g$ that tend to a point in the boundary. For $\tau<0$ the system \cref{eqn:cogeo-tau-zeta} is equivalent to \cref{eqn:cogeo-tau}, and hence equivalent to the cogeodesic flow (on the open subset of $T^*X$ where the Fermi coordinates are defined and $\xi_0>0$) provided the energy condition holds. When a solution of any of these systems satisfies the energy condition $2H=1$ at some $\tau<0$ (or some $t$) then it satisfies $2H=1$ for all $\tau<0$ (or all $t$). Moreover, when the energy condition holds for $\tau<0$ it clearly also holds at $\tau=0$, where it is equivalent to the condition $\zeta_0|_{\tau=0}=1$ (cf.\ \cref{eq:ES-zeta}). Our first lemma shows that, conversely, by specifying that $\zeta_0=1$ at $\tau=0$ one also ensures that the energy condition holds for $\tau<0$.
\begin{lemma}\label{lem:energy-cond-extends-in}
If $(x^{\alpha},\zeta_0,\zeta_{\alpha})$ is a solution of \cref{eqn:cogeo-tau-zeta} for $\tau \in [-\delta_0,0]$ with $\zeta_0|_{\tau=0}=1$, then the energy condition \cref{eq:ES-zeta} 
holds for all $\tau \in [-\delta_0,0]$, where $\xi_{\alpha}$ is given by $\zeta_{\al}+A_{\al}$ with $A_{\al}$ as in \cref{eq:A}.
\end{lemma}
\begin{proof} 
Fix a Fermi coordinate system. Note that for $x^0<0$ and $\zeta_0>0$ the variables $(\zeta_0,\zeta_{\alpha})$ are smoothly related to $(\xi_0,\xi_{\alpha})$ and can be regarded as nonstandard fiber coordinates for $T^*X$. If we augment the system \cref{eqn:cogeo-tau-zeta} with the trivial equation $\frac{dx^0}{d\tau}=1$ then, in the open subset of $T^*X$ where the Fermi coordinates are defined and $\xi_0>0$ (equivalently, $\zeta_0>0$), we may think of the augmented system as describing the flow of a vector field on $T^*X$ of the form $\frac{\partial}{\partial x^0} + V$, where $V$ depends on $x^0$ but has no $\frac{\partial}{\partial x^0}$-component. In the process of deriving the system \cref{eqn:cogeo-tau-zeta} from the Hamiltonian system \cref{eq:orig-cogeod-flow} we have made use of the energy surface equation and then have changed independent variable from $t$ to $\tau$ (with $\frac{d\tau}{dt} = \rho \zeta_0$). It follows that on the energy surface $2H=1$ the vector field $\frac{\partial}{\partial x^0} + V$ is $(\rho\zeta_0)^{-1}$ times the Hamiltonian vector field $V_{H}$ of $H$. In particular, $\frac{\partial}{\partial x^0} + V$ is tangent to the energy surface $2H=1$ when $x^0<0$.
Fixing $(x^0,x^{\alpha}, \zeta_0,\zeta_{\alpha})$ as our preferred coordinates on $T^*X$, the vector field  $\frac{\partial}{\partial x^0} + V$ extends continuously to $x^0=0$ (though not as a vector field on $T^*\oX$, cf. \cref{rem:irreg-coords} below). Since 
\begin{equation}
H= \frac{1}{2}\left((\zeta_0)^2 + \rho^2h^{\al\be}(\zeta_{\al}+A_{\al})(\zeta_{\be}+A_{\be})\right)
\end{equation}
is $\cC^1$ up to $x^0=0$ in this coordinate system, it follows that $\frac{\partial}{\partial x^0} + V$ remains tangent to the level set $2H=1$ up to $x^0=0$. Therefore, an integral curve of $\frac{\partial}{\partial x^0} + V$ that lies on the energy surface $2H=1$ at one point (with $x^0\leq 0$) must lie entirely in the energy surface. Note that at $x^0=0$ the energy condition $2H=1$ is equivalent to $\zeta_0=1$ ($\zeta_0=-1$ is ruled out because we require $\zeta_0>0$). Since solutions of \cref{eqn:cogeo-tau-zeta} correspond to the integral curves of $\frac{\partial}{\partial x^0} + V$ with $x^0(\tau)=\tau$, the result follows.
\end{proof}

\begin{remark}\label{rem:irreg-coords}
Note that when considering the extension of $\frac{\partial}{\partial x^0} + V$ to $x^0=0$ in the above argument we are forgetting about the cotangent bundle and working only in the $(x^0,x^{\alpha}, \zeta_0,\zeta_{\alpha})$-coordinates. One can consider the extended vector field as being defined on a vector bundle over $\oX$ (near $\partial X$) but this bundle will not be $T^*\oX$. In the asymptotically hyperbolic case we are implicitly extending to the \emph{$b$-cotangent bundle} (in the terminology of Melrose \cite{Melrose-APS}), but when asymptotic hyperbolicity fails this is no longer the case due to the logarithmically divergent term in the relationship \cref{def-eq:zeta-al} between the tangential momentum variables $\xi_{\al}$ and $\zeta_{\al}$.
\end{remark}

\cref{lem:energy-cond-extends-in} ensures that we obtain solutions to the cogeodesic flow  for $g$ near $\partial X$ from solutions of \cref{eqn:cogeo-tau-zeta} with $\zeta_0|_{\tau=0}=1$. We state this as another lemma:
\begin{lemma}\label{lem:geodesics-from-w-and-tau-system}
If $(x^{\alpha},\zeta_0,\zeta_{\alpha})$ is a solution of \cref{eqn:cogeo-tau-zeta} for $\tau \in [-\delta_0,0]$ with $\zeta_0|_{\tau=0}=1$, then the curve $\tilde{\gamma}:[-\delta_0,0)\to X$ defined in the given Fermi coordinate system by $\tilde{\gamma}(\tau) = (\tau, x^{\alpha}(\tau))$ is a geodesic for $g$ reparametrized in terms of minus the $h$-distance from $\partial X$.
\end{lemma}
\begin{proof}
Fix a Fermi coordinate chart for $\partial X$ and suppose $(x^{\alpha},\zeta_0,\zeta_{\alpha})$ is a solution of \cref{eqn:cogeo-tau-zeta} for $\tau \in [-\delta_0,0]$ with $\zeta_0|_{\tau=0}=1$. Then, by \cref{lem:energy-cond-extends-in}, $(x^{\alpha},\zeta_0,\zeta_{\alpha})$ satisfies the energy condition $2H=1$ for all $\tau$. Hence, setting $\xi_{\al}=\zeta_{\al}+A_{\al}$ we obtain a solution of \cref{eqn:cogeo-tau} satisfying the energy condition \cref{eq:ES-zeta} for $\tau \in [-\delta_0,0)$. Using that $\rho^{2}h^{\al\be}\xi_{\al}\xi_{\be} = 1-(\zeta_0)^2$ it is straightforward to show that 
\begin{equation}
    \left|\frac{d}{d\tau}\tilde{\gamma}\right|_g^2 = \rho^{-2}\left(1+ \frac{\rho^2h^{\al\be}\xi_{\al}\xi_{\be}}{(\zeta_0)^2}\right) = \frac{1}{(\rho \zeta_0)^2},
\end{equation}
and it follows that if we introduce an arclength parameter $t$ for $\tilde{\gamma}$ with the same orientation as the parameter $\tau$ then $\frac{d\tau}{dt} = \rho \zeta_0$. For convenience, we fix $t$ by requiring that $t=0$ when $\tau=-\delta_0$. Considering $x^0$ and the solution $(x^{\alpha},\zeta_0,\xi_{\alpha})$ of \cref{eqn:cogeo-tau} as functions of $t$ rather than $\tau$ we therefore obtain a solution of \cref{eqn:cogeo:initial-2} satisfying \cref{eq:ES-zeta}. This is, of course, equivalent to a solution of \cref{eq:cog-flow-ES} satisfying \cref{es} and hence an integral curve of the cogeodesic flow \cref{eq:orig-cogeod-flow} corresponding to a unit-speed geodesic. Clearly this geodesic is the arclength reparametrization $\gamma:[0,\infty)\to X$ of $\tilde{\gamma}$ by $t$. This proves the result. 
\end{proof}

From \cref{lem:geodesics-from-w-and-tau-system} we immediately obtain the following useful lemma:
\begin{lemma}
For every point $q\in \partial X$ there is a geodesic $\gamma:[0,\infty)\to X$  for $g$ with $\lim_{t\to\infty} \gamma(t) = q$.
\end{lemma}
\begin{proof}
Given $q\in \partial X$ we fix a Fermi coordinate system containing $q$ and consider the initial value problem for the system \cref{eqn:cogeo-tau-zeta} with ``initial'' condition at $\tau =0$ given by requiring $x^{\alpha}|_{\tau=0}$ to be the coordinates of $q$, $\zeta_0|_{\tau=0}=1$ and $\zeta_{\al}|_{\tau=0}=0$ (the choice of $\zeta_{\al}|_{\tau=0}$ does not matter). There is then $\delta_0>0$ and a solution of \cref{eqn:cogeo-tau-zeta} for $\tau\in [-\delta_0,0]$ with these ``initial'' conditions, and $\tilde{\gamma}:[-\delta_0,0)\to X$ defined by $\tilde{\gamma}(\tau) = (\tau, x^{\alpha}(\tau))$ is then the required geodesic (up to reparametrization). This proves the lemma.
\end{proof}

\subsection{Proof of Theorem \ref{thm:AH-characterization}} In this section we prove \cref{thm:AH-characterization}, which characterizes asymptotically hyperbolic manifolds in terms of the boundary regularity of geodesics. Let $(X, g)$ be a conformally compact Riemannian manifold with $g= \rho^{-2}h$. Before we prove the theorem, we establish one more lemma: 

\begin{lemma}\label{lem:asymp-xalphadot}
Let $\gamma: [0, \infty) \to X$ be a geodesic for $g$ that tends to a point $\gamma_{\infty}$ on $\partial X$ and let $(x^0,x^1, \dots, x^n)$ be a Fermi coordinate chart for the boundary $\partial X$ with respect to the metric $h$ in a neighborhood of $\gamma_{\infty}$. Then, near  $\gamma_{\infty}$, we have
\begin{equation}\label{eq:x-alpha-dot-more-precise-asymp}
     \dot{x}^{\alpha} = \frac{\kappa^{\alpha}}{\kappa^2}\cdot \rho^2\log{\rho} + O(\rho^2),
\end{equation}
where $\kappa_{\alpha}=\partial_{\alpha}\kappa$ and $\kappa^{\alpha} = h^{\alpha\beta}\kappa_{\beta}$ along $\partial X$. In particular, the coefficient of $\rho^2\log{\rho}$ in \cref{eq:x-alpha-dot-more-precise-asymp} depends only on the boundary endpoint $q=\gamma_{\infty}$ and not on the particular geodesic $\gamma$ tending to $q$. Moreover, the $\rho^2\log{\rho}$-coefficient vanishes at $q$ if and only if $d\kappa|_q=0$.
\end{lemma}
\begin{proof}
Recall from \cref{sec:extending} that $\dot{x}^{\alpha} = \rho^2 h^{\al\be}\xi_{\be} =\rho^2  h^{\al\be}(\zeta_{\be}+A_{\be})$. From the definition \cref{eq:A} of $A_{\beta}$ we therefore have 
\begin{equation}
\dot{x}^{\alpha} = \rho^2  h^{\al\be}\zeta_{\be}+ \frac{1}{\zeta_0} \cdot \frac{\kappa^{\alpha}}{\kappa^2} \cdot  \rho^2 \log |x^0|.
\end{equation}
Recall also that $\zeta_0=1+O(\rho)$ and that  by \cref{lem:ext-to-tau-eq-0} $(x^{\al}(\tau),\zeta_0(\tau), \zeta_{\al}(\tau))$ is $\mathcal{C}^1$ up to $\tau =0$, where $\tau = x^0$, so that the functions $\zeta_{\be}$ are bounded. It follows that
\begin{equation}
     \dot{x}^{\alpha} = \frac{\kappa^{\alpha}}{\kappa^2}\cdot \rho^2\log|x^0| + O(\rho^2).
\end{equation}
Since $\log|x^0| = \log\rho+O(1)$ the result follows.
\end{proof}

We now prove \cref{thm:AH-characterization}. Suppose $\partial X$ is connected. That asymptotic hyperbolicity (condition (i) in \cref{thm:AH-characterization}) implies (ii) follows from \cref{lem:asymp-xalphadot}, and that (i) implies (iii) follows easily from the regularity of the system \cref{eqn:cogeo-tau} in the asymptotically hyperbolic case (cf.\ \cite{Mazzeo-Thesis,ChenHassell-I}). That (ii) implies (i) follows from \cref{eq:x-alpha-dot-more-precise-asymp} and the fact that there is a geodesic tending to every boundary point; if (ii) holds then $d\kappa =0$ on $\partial X$ and hence $\kappa$ is constant.
It follows that (ii) also implies (iii).

To see that (iii) implies (ii), suppose (ii) does not hold. Then, again by \cref{eq:x-alpha-dot-more-precise-asymp}, there is a point $q$ and a geodesic $\gamma:[0,\infty)\to X$ tending to $q$ such that $d\kappa|_q \neq 0$. Moreover, $\dot{x}^{\al} = \kappa^{-2}\kappa^{\alpha}|_{q}\cdot \rho^2\log{\rho} + O(\rho^2)$ implies that $\frac{dx^{\alpha}}{d\tau} = \kappa^{-2}\kappa^{\alpha}|_q\cdot\rho\log{\rho} + O(\rho)$, which implies that $x^{\alpha}(\tau)$ is not smooth up to $\tau=0$. Hence, if (ii) does not hold, then (iii) does not hold. Thus (iii) implies (ii). This concludes the proof of \cref{thm:AH-characterization}.

\subsection{Proof of Theorem \ref{thm:boundary-exponential}}

Here we prove \cref{thm:boundary-exponential} from the introduction.

We first prove (i), stating that if $p \in X$ is sufficiently close to $\partial X$ then there is an open subset $V_p$ in the unit tangent space $S_p X \subset T_p X$ such that all geodesics in $(X,g)$ with initial direction in $V_p$ tend to a point in the boundary $\partial X$. For this we only require the results of \cref{subsec:prelim-obs}. If $p\in U$ then, taking $V_p \subseteq S_p X$ to be all unit tangent vectors with positive $\frac{\partial}{\partial x^0}$-component, we know from \cref{prop:elem-obs,prop:def-bdry-pt} that any geodesic $\gamma:[0,\infty)\to X$ with $\dot{\gamma}(0)\in V_p$ tends to a definite point in the boundary $\partial X$ as $t\to\infty$. This proves (i).

With $V_p$ thus defined for $p\in U$ we let $\exp_{p,\infty}:V_p\to \partial X$ denote the map that takes the initial velocity $\dot{\gamma}(0)\in V_p$ of a geodesic $\gamma:[0,\infty)\to X$ emanating from $p$ to the endpoint $\gamma_{\infty} = \lim_{t\to\infty}\gamma(t) \in \partial X$. Before we prove part (ii) of \cref{thm:boundary-exponential} we establish another lemma.
\begin{lemma}\label{lem:rho-t-asymp}
    Let $\gamma:[0,\infty)\to X$ be a geodesic for $g$ with $\gamma(0)\in U$ and $\dot{x}^0(0)>0$. Then $\rho \sim e^{-\kappa t}$ along $\gamma$, where $\kappa$ is evaluated at $q=\gamma_{\infty}$. That is, there are constants $C_1,C_2>0$ such that $C_1 e^{-\kappa(q) t} \leq \rho(\gamma(t)) \leq C_2 e^{-\kappa(q) t}$ for all $t\in [0,\infty)$.
\end{lemma}
\begin{proof}
    Since $\rho/x^0$ is negative, bounded and bounded away from zero on $U$ it suffices to show that $x^0\sim - e^{-\kappa t}$. To see this we recall that (with $\tau=x^0$) we have $\frac{dt}{d\tau} = \frac{1}{\rho \zeta_0} = -\frac{1}{\kappa \tau} + O(1) = -\frac{1}{\kappa \tau} + O(1)$,  where $\kappa$ is evaluated at $q=\gamma_{\infty}$. Hence $-\kappa \frac{dt}{d\tau} = \frac{1}{\tau} + O(1)$, so that $-\kappa t = \log |\tau| + O(1)$. It follows that $|\tau| \sim e^{-\kappa t}$ and hence $\tau\sim - e^{-\kappa t}$. This proves the result.
\end{proof}

We now establish part (ii) of \cref{thm:boundary-exponential} by showing that for each $p\in U$ the map $\exp_{p,\infty}$ is a local $\cC^{\infty}$-diffeomorphism. That $\exp_{p,\infty}$ is $\cC^{\infty}$-smooth follows from \cref{thm:asymptotic-behavior-of-geodesic}(iv). Moreover, from the results above, for a unit vector $v$ orthogonal to $v_0\in V_p\subseteq S_pX$ it follows that if $\gamma_s:[0,\infty)\to X$ is the family of geodesics with initial velocity $\cos(s)v_0 + \sin (s)v\in V_p$ defined for $s$ near $0$ and $\gamma_{\infty}(s)=\gamma_{s,\infty}$ is the curve in $\partial X$ traced by the boundary endpoints of the curves $\gamma_s$ then
\begin{equation}
    \left.\frac{d}{ds}\right|_{s=0} \gamma_{\infty}(s) = \lim_{t\to\infty} J(t),
\end{equation}
where $J(t) = \left.\frac{d}{ds}\right|_{s=0} \gamma_s(t)$ is the Jacobi field associated with the variation $\gamma_s$ of the geodesic $\gamma = \gamma_0$. To show that $\exp_{p,\infty}$ is a local diffeomorphism it suffices to show that $\left.\frac{d}{ds}\right|_{s=0} \gamma_{\infty}(s)\neq 0$ when $v\neq 0$. Note that, by \cref{eq:curvature-asymp}, the maximum $K_1(t)$ and minimum $K_2(t)$ of the sectional curvatures at $\gamma(t)$ are both of the form $-\kappa^2 + O(\rho) = -\kappa^2 + O(e^{-\kappa t})$, where $\kappa$ is evaluated at $\gamma_{\infty}$. If $(X,g)$ were the hyperbolic space with constant sectional curvature $-\kappa^2$ then the length of the Jacobi field $J(t)$ with respect to $g$ would be $\kappa^{-1}\sinh (\kappa t)|v|_g$. In the general case it follows from the Rauch comparison theorem that 
\begin{equation}
    |J(t)|_g \sim e^{\kappa t}|v|_g
\end{equation}
for $t$ large. Hence, by \cref{lem:rho-t-asymp}, 
\begin{equation}
    |J(t)|_h = \rho |J(t)|_g \sim |v|_g
\end{equation}
for $t$ large. In particular, when $v\neq 0$ the length $|J(t)|_h$ of the Jacobi field is bounded away from zero for large $t$ and hence $ \left.\frac{d}{ds}\right|_{s=0} \gamma_{\infty}(s) = \lim_{t\to\infty} J(t)\neq 0$. It follows that the differential of $\exp_{p,\infty}$ is injective at $v_0$. Since $v_0$ was arbitrary, this proves part (ii) of \cref{thm:boundary-exponential}.

It remains to prove part (iii), which states that for every point $p_{\infty}\in \partial X$ there is a point $p\in U$ and a subset $V_{p,p_{\infty}}$ of $V_p$ such that $\exp_{p,\infty}$ restricts to a diffeomorphism from $V_{p,p_{\infty}}$ to a neighborhood of $p_{\infty}$ in $\partial X$. This follows easily from part (ii) and \cref{lem:geodesics-from-w-and-tau-system},  since for every point $p_{\infty}\in \partial X$ there is a point $p\in U$ and a geodesic $\gamma:[0,\infty)\to X$ with $\dot{x}^0(0)>0$ such that $\lim_{t\to\infty} \gamma(t) = p_{\infty}$. Without loss of generality, we can assume that $\gamma(0)\in U$ and $\dot{x}^0(0)>0$ (so that $\dot{\gamma}(0)\in V_p$ by definition). Therefore, being a local diffeomorphism, the map $\exp_{p,\infty}:V_p\to \partial X$ must restrict to a diffeomorphism from some neighborhood $V_{p,p_{\infty}}$ of $\dot{\gamma}(0)$ to a neighborhood of $p_{\infty}$. This proves (iii).

This concludes the proof of \cref{thm:boundary-exponential}.

\subsection{Proof of Corollary \ref{cor:smooth-boundary-extension}} \label{sec:proof-of-cor}

Here we give a proof of \cref{cor:smooth-boundary-extension}. Before we give the proof it is useful to note that if $(X,g)$ is a conformally compact manifold, then the topological manifold structure of $\oX=X\cup \partial X$ is determined by the geometry of $(X,g)$. Points in $\partial X$ can be thought of as equivalence classes of geodesics $\gamma:[0,\infty)\to X$ in $(X,g)$ that tend to the same endpoint in $\partial X$ (this can be characterized intrinsically using the notion of \textit{asymptotic} geodesics from \cite{EberleinO'Neill1973}). One can also construct $\cC^0$ coordinate charts for $\oX$ in a neighborhood of any point $p_{\infty}\in \partial X$ intrinsically from the geometry of $(X,g)$. We briefly sketch one way to do this: Fix a unit-speed geodesic $\gamma:[0,\infty)\to X$ tending to $p_{\infty}$ such that $\gamma(0)\in U$ and $\dot{x}^0(0)>0$ and set $p=\gamma(0)$. Let $V_{p,p_{\infty}}\subset V_p$ be a neighborhood of $\dot{\gamma}(0)$ such that the map $V_{p,p_{\infty}}\times (0,\infty)\to X$ defined by $(v,t)\mapsto \exp_p(tv)$ is injective. Viewing $(v,r=e^{-t})\in V_{p,p_{\infty}}\times (0,1)$ as coordinates for the image of this map and extending $r$ to $\partial X$ by zero (and shrinking $V_{p,p_{\infty}}$ if necessary), the coordinates $(v,r)$ extend to a $\cC^0$ coordinate system for $\oX$ containing $p_{\infty}$ with $(v,r)\in V_{p,p_{\infty}}\times [0,1)$. 
\begin{remark}
Note that, geometrically, it would be better to take $r=e^{-\kappa t}$ along each geodesic with initial velocity $v$, where $\kappa$ is evaluated at the endpoint of the geodesic, since then $r\sim \rho$ (as opposed to a variable power of $\rho$) but this does not matter at present since we only require $\cC^0$ regularity. We will return to this point in \cref{sec:normal-reg} below.
\end{remark}

From the above discussion it is clear that an isometry between two conformally compact manifolds must extend continuously to the boundary (since it clearly takes such a coordinate system for the domain to another such coordinate system for the target, and when viewed with respect to these two $\cC^0$ coordinate systems therefore becomes the trivial map $(v,r)\mapsto (v,r)$ for $r>0$ which is, of course, continuous up to the boundary $r=0$). The continuous extension is clearly also a homeomorphism.

To prove \cref{cor:smooth-boundary-extension} it remains to show that if $(X_1,g_1)$ and $(X_2, g_2)$ are conformally compact Riemannian manifolds and $F$ is an isometry from $(X_1, g_1)$ to $(X_2, g_2)$ then the continuous extension $F:\oX_1\to\oX_2$ restricts to a $\cC^{\infty}$-smooth conformal diffeomorphism $f:\partial X_1 \to \partial X_2$. The smoothness of the map $f$ follows easily from \cref{thm:boundary-exponential}: If $p_{\infty}\in \partial X_1$ then there is $p\in X_1$ such that $\exp^{g_1}_{p,\infty}:V_{p,p_{\infty}}\to \partial X_1$ gives a smooth parametrization of a neighborhood of $p_{\infty}$. Since $F$ is an isometry that extends to a homeomorphism $\oX_1\to\oX_2$ it follows that if $p'=F(p)$, $p'_{\infty} = f(p_{\infty})$ and $V_{p',p'_{\infty}}=dF_{p}(V_{p,p_{\infty}})$ then $\exp^{g_2}_{p',\infty}:V_{p',p'_{\infty}}\to \partial X_2$ gives a smooth parametrization of a neighborhood of $p'_{\infty}=f(p_{\infty})$. Moreover, since $f$ takes the endpoint $\gamma_{\infty}$ of a geodesic $\gamma$ in $X_1$ to the endpoint $\gamma'_{\infty}$ of the geodesic $\gamma'=F\circ \gamma$ in $X_2$, near $p_{\infty}$ we have
\begin{equation}
f = \exp^{g_2}_{p',\infty} \,\circ \; dF_p \,\circ\,(\exp_{p,\infty}^{g_1})^{-1},
\end{equation}
which is clearly smooth. Hence $f:\partial X_1 \to \partial X_2$ is $\cC^{\infty}$.

It now remains only to show that $f:\partial X_1 \to \partial X_2$ is a conformal diffeomorphism. One way to see this is as follows: Let $p_{\infty} \in \partial X_1$ and let $\gamma:[0,\infty)\to X_1$ be a geodesic tending to $p_{\infty}$ with $\gamma(0)=p\in U$ and $\dot{x}^0(0)>0$. Let $v_0=\dot{\gamma}(0)$ and let $(v_1,\ldots, v_n)$ be an orthonormal frame for $\dot{\gamma}(0)^{\perp}\subset T_p X_1$ and for each $\alpha\in \{1,\ldots,n\}$ let $J_{\alpha}$ be the Jacobi field along $\gamma$ with $J_{\alpha}(0)=0$ and $\nabla_{\dot{\gamma}(0)}J_{\alpha} = v_{\alpha}$. Then $J_1(t),\ldots, J_n(t)$ are linearly independent for all $t>0$. Moreover, $J_{\alpha,\infty}= \lim_{t\to \infty}J_{\alpha}(t)$ is equal to $d  \exp^{g_1}_{p,\infty}|_{v_0}(v_{\alpha})$ and $(J_{1,\infty},\ldots, J_{n,\infty})$ is a basis for $T_{p_{\infty}}\partial X_1$. Now, since $F$ is an isometry, it takes $\gamma$ to a geodesic $\gamma'$ and each $J_{\alpha}$ to a Jacobi field $J'_{\alpha}$ along $\gamma'$ such that $J'_1(t),\ldots, J'_n(t)$ are linearly independent for all $t>0$. Moreover, $\gamma'$ tends to $p'_{\infty} = f(p_{\infty})$ as $t\to\infty$ and if $J'_{\alpha,\infty}= \lim_{t\to \infty}J'_{\alpha}(t)$ then $(J'_{1,\infty},\ldots, J'_{n,\infty})$ is a basis for $T_{p'_{\infty}}\partial X_2$ and is the image of $(J_{1,\infty},\ldots, J_{n,\infty})$ under the map $df_{p_{\infty}}$. Writing $g_1 = \rho_1^{-2}h_1$ and $g_{2}=\rho^{-2}_2h_2$, where the $h_i$ are smooth and nondegenerate up to the respective conformal infinities, and setting  $h_{\alpha\beta}= h_1(J_{\alpha},J_{\beta})$ and $h'_{\alpha\beta}= h_2(J'_{\alpha},J'_{\beta})$, since $F^*g_2=g_1$ we have that 
\begin{equation}\label{eq:conformality-on-gamma-perp}
    h'_{\al\be}(t) = \lambda(t)h_{\al\be}(t),
\end{equation}
for all $t\in (0,\infty)$. Finally, by continuity, $h_{\al\be}^{\infty}=\lim_{t\to\infty}h_{\al\be}(t) = h_1(J_{\al,\infty}, J_{\be,\infty})$ and ${h'}_{\al\be}^{\infty}=\lim_{t\to\infty}h'_{\al\be}(t) = h_2(J'_{\al,\infty}, J'_{\be,\infty})$, and since the matrices $(h_{\al\be}^{\infty})$ and $({h'}_{\al\be}^{\infty})$ are both nondegenerate, it follows from \cref{eq:conformality-on-gamma-perp} that 
\begin{equation}\label{eq:conf-at-bdry}
    {h'}_{\al\be}^{\infty} = \lambda_{\infty}h_{\al\be}^{\infty},
\end{equation}
for some $\lambda_{\infty} >0$. But, since $df_{p_{\infty}}$ takes the frame $(J_{\al,\infty})$ to the frame $(J'_{\al,\infty})$, equation \cref{eq:conf-at-bdry} is equivalent to
\begin{equation}
    f^*h_2 = \lambda_{\infty} h_1
\end{equation}
at $p_{\infty}$. Since the point $p_{\infty}\in \partial X_1$ was arbitrary, we conclude that $f$ is a conformal diffeomorphism. This proves \cref{cor:smooth-boundary-extension}.

\subsection{Regularity in the normal direction}\label{sec:normal-reg}

In this section we supplement \cref{cor:smooth-boundary-extension} with the following boundary regularity result.  

\begin{theorem}\label{thm:boundary-extension-nor-reg}
    Let $(X_1,g_1)$ and $(X_2, g_2)$ be conformally compact Riemannian manifolds. Let $F$ be an isometry from $(X_1, g_1)$ to $(X_2, g_2)$. Then $F$ extends to a $\cC^1$ diffeomorphism from $\oX_1$ to $\oX_2$. If $(X_1,g_1)$ is asymptotically hyperbolic, then the extension is $\cC^{\infty}$.
\end{theorem}
\begin{remark}
(i) Note that, in contrast to our previous results, the result that the extension is $\cC^1$ in the general case may not be optimal. Since our proof makes use of coordinates constructed from the solutions of our extended flow \cref{eqn:cogeo-tau-zeta}, which are only $\cC^1$ up to the boundary in general, we cannot establish higher boundary regularity by this method. Nevertheless, it is useful to know that the extended map $F:\oX_1\to\oX_2$ is at least $\cC^1$ up to the boundary. (ii) We could have used this result to establish the conformality of $f$ in the previous section, but we wanted to emphasize that \cref{cor:smooth-boundary-extension} follows directly from \cref{thm:boundary-exponential} (and also to separate out the discussion of tangential and normal regularity at the conformal infinity). (iii) The result that $F$ extends to a $\cC^1$ diffeomorphism from $\oX_1$ to $\oX_2$ continues to hold under the weaker regularity hypotheses on the metrics $g_i$ (or more precisely on their conformal compactifications $h_i$) introduced in \cref{rem:polyhom}.
\end{remark}

To prove \cref{thm:boundary-extension-nor-reg} we modify the canonical $\cC^0$ boundary coordinate charts constructed in \cref{sec:proof-of-cor} to improve the regularity. As before, given a point $p_{\infty}\in \partial X$ we fix a point $p\in U$ and a subset $V_{p,p_{\infty}}\subset V_p$ such that the map $V_{p,p_{\infty}}\times (0,\infty] \ni (v,t) \mapsto \exp_p(tv) \in \oX$ (extended continuously to $t=\infty$ in the obvious way) parametrizes a neighborhood $N$ of the point $p_{\infty}$ in $\oX$. Identifying $V_{p,p_{\infty}}$ with its image in $\partial X$ under $\exp_{p,\infty}$ we may think of $\kappa$ as a $\cC^{\infty}$ function of $v\in V_{p,p_{\infty}}$. With $\kappa$ interpreted this way we define the function $r$ on $N$ by $r=e^{-\kappa t}$ when $t\in [0,\infty)$ and $r=0$ when $t=\infty$. The new coordinates $(v,r)\in V_{p,p_{\infty}}\times [0,1)$ for $N$ are $\cC^{\infty}$ away from $\partial X$ ($r=0$). Up to the boundary we have:
\begin{lemma}\label{lem:local-def-fun}
The smooth coordinates $(v,r=e^{-\kappa t}) \in V_{p,p_{\infty}}\times (0,1)$ for $N\setminus\partial X$ extend to $\cC^1$ coordinates $(v,r) \in V_{p,p_{\infty}}\times [0,1)$ for $N$ with $r$ a local defining function for $\partial X$. If $(X,g)$ is asymptotically hyperbolic then the extension is $\cC^{\infty}$.
\end{lemma}
\begin{proof}
We assume, without loss of generality, that $N$ is contained in a single Fermi coordinate chart $(x^0,x^1,\ldots, x^n)$.  
For each $v\in  V_{p,p_{\infty}}$, let $\gamma_v(t)=\exp_p(tv)$ be the corresponding geodesic (for $t\geq 0$) and let $\tilde{\gamma}_v(\tau)$ be its reparametrization in terms of minus the $h$-distance from the boundary ($\tau\in [-\delta_0,0)$). Writing the $(x^0,x^1,\ldots, x^n)$-coordinates of $\tilde{\gamma}_v(\tau)$ as $(x^0(v,\tau), \ldots, x^n(v,\tau))$ gives a $\cC^{\infty}$-diffeomorphism
\begin{equation}\label{map:v-tau-to-xi}
 V_{p,p_{\infty}}\times (-\delta_0,0)   \ni (v,\tau) \mapsto  (x^0(v,\tau), \ldots, x^n(v,\tau)) \in N\setminus \partial X.
\end{equation}
By our previous discussion of the system \cref{eqn:cogeo-tau-zeta} and the regularity results of \cref{app:reg}, the map \cref{map:v-tau-to-xi} extends to $\tau=0$ as a $\cC^1$-diffeomorphism $ V_{p,p_{\infty}}\times (-\delta_0,0]\to N$ (and in the asymptotically hyperbolic case this map is $\cC^{\infty}$). We may therefore view $(v,\tau)$ as a $\cC^1$ coordinate system on $N$. The result will follow (in the general case) if we can show that $r$ is a $\cC^1$ function of $(v,\tau)$ up to $\tau=0$. 

Viewing $v$ as a parameter, we write $\partial_{\tau} r = \partial_0 r$ as $r'$. Along each geodesic $\frac{dt}{d\tau} = \frac{1}{\rho \zeta_0}$, and hence for $\tau<0$ we have
\begin{equation}
    r' = \frac{dr}{dt}\cdot \frac{dt}{d\tau} = \frac{-\kappa r}{\rho \zeta_0},
\end{equation}
where $\kappa$ is a smooth function of $v$, and $\rho$ and $\zeta_0$ are viewed as smooth functions of $(v,\tau)$ that are $\cC^1$ up to $\tau=0$. For $\tau<0$ we define the smooth function $\omega$ of $(v,\tau)$ by requiring $r=-e^{\omega}\tau$. Then $r'= -\omega'e^{\omega}\tau - e^{\omega}$, where $\omega' = \partial_\tau \omega$, and it follows easily that
\begin{equation}
    \omega' = \frac{1}{\tau}\left(\frac{-\kappa \tau}{\rho \zeta_0} - 1 \right)=:\varphi,
\end{equation}
for $\tau<0$. Now, $\rho = -\kappa \tau + \beta \tau^2$, where $\beta$ is a $\cC^{\infty}$ function of the Fermi coordinates (up to $\partial X$) and hence is $\cC^1$ up to $\tau=0$ as a function of $(v,\tau)$. Similarly, $\zeta_0 = 1+\tau \alpha$, where $\alpha$ is continuous up to $\tau=0$. It follows that 
\begin{equation}
    \rho \zeta_0 = -\kappa \tau + \sigma \tau^2,
\end{equation}
where $\sigma = -\kappa \alpha + \beta + \alpha \beta\tau$ is continuous up to $\tau=0$. Hence 
\begin{equation}
    \varphi = \frac{1}{\tau}\left( \frac{\rho \zeta_0 - \sigma \tau^2}{\rho \zeta_0} - 1\right) = -\frac{\sigma \tau}{\rho \zeta_0}.
\end{equation}
Since $\sigma/\zeta_0$ is continuous up to $\tau=0$ and $\rho/\tau = -\kappa +\beta\tau$ is continuous up to $\tau=0$ and nonvanishing, it follows that  $\varphi$ is continuous up to $\tau=0$. Hence $\omega$ is continuous on $N$. Since $r(v,-\delta_0)= e^{-\kappa t}|_{t=0}=1$ for all $v$, we have $r(v,\tau) = 1+ \int_{-\delta_0}^{\tau}\left.(-\varphi e^{\omega}\varsigma - e^{\omega})\right|_{(v,\varsigma)}\,d\varsigma$ on $N$. Hence $\partial _{\tau} r$ exists on $N$ (not just $N\setminus \partial X$) and is continuous up to $\tau=0$. 

In the general setting, therefore, it remains only to consider the first order $v$-partial derivatives. For clarity, we identify $V_{p,p_{\infty}}$ with its image in $\partial X$ under $\exp_{p,\infty}$, so that the Cartesian coordinates $(x^1,\ldots,x^n)$ can be used for $v$ when defining partial derivatives. Now, the function $r$ is smooth for $\tau<0$. Using \cref{thm:appendix-ode-reg-Ck} and the preceding expressions for
$\rho$, $\zeta_0$, $\varphi$, and $r$, one can show that the first
$v$-partial derivatives of $r$ tend to zero as $\tau\to0^-$, and hence
extend continuously by zero to $\tau=0$. Since $r=0$ when $\tau=0$ and $\partial_{\tau}r$ is continuous up to $\tau=0$, it follows that $r$ is a $\cC^1$ function of $(v,\tau)$ on $N$.

To see that $(v,r)$ extend smoothly to $\partial X$ in the asymptotically hyperbolic case, we note that in this case $(v,\tau)$ are $\cC^{\infty}$ coordinates on $N$ and the functions $\varphi$ and $\omega$ are $\cC^{\infty}$ up to $\partial X$ so that  $r(v,\tau)$ is  $\cC^{\infty}$ on $N$.
\end{proof}

\cref{thm:boundary-extension-nor-reg} then clearly follows from \cref{lem:local-def-fun} by the same argument that was used to establish the continuous extension to the boundary in \cref{sec:proof-of-cor}. An isometry $F:X_1\to X_2$ between conformally compact manifolds will take a coordinate system of the form appearing in \cref{lem:local-def-fun} to another of the same kind, and when viewed in these coordinates becomes the trivial map  $(v,r)\mapsto (v,r)$, for $r>0$. Since the coordinates extend to be $\cC^1$ up to the boundary ($\cC^{\infty}$ in the asymptotically hyperbolic case), the map $F$ trivially extends with the same regularity.

\subsection{Proof of Theorem \ref{thm:shooting-geodesics-from-boundary}} Here we prove \cref{thm:shooting-geodesics-from-boundary}, which includes the asymptotic expansion \cref{eq:geodesic-asymp} for a geodesic when viewed as a function of $x=-x^0$ in a Fermi coordinate chart.

Let $q\in \partial X$ and let $(x, y^1, \dots, y^n)$ be a Fermi coordinate chart adapted to $\partial X$ in $\oX$ with respect to the metric $h$, with $x > 0$ in $X$ and centered at $q$. In the notation that we have been using more frequently, we have $(x^0,x^1,\ldots,x^n) = (-x, y^1, \dots, y^n)$. We wish to show that, as unparametrized curves, the geodesics approaching $q\in \partial X$ have asymptotic expansion \cref{eq:geodesic-asymp} and are parametrized by $u\in T_q\partial X$. Since we know that we can obtain the geodesics (near $\partial X$) by solving \cref{eqn:cogeo-tau-zeta} for $\tau\leq 0$ near $0$ with $x^{\alpha}|_{\tau=0}=0$, $\zeta_0|_{\tau=0}=1$ and $\zeta_{\alpha}|_{\tau=0}$ prescribed freely (see \cref{sec:geod-from-infty}), the proof boils down to establishing the asymptotic expansion of the form \cref{eq:geodesic-asymp} and showing that the $u^{\alpha}$ appearing in that expansion are in one-to-one correspondence with the $\zeta_{\alpha}|_{\tau=0}$. \cref{thm:shooting-geodesics-from-boundary} will therefore follow easily from the following lemma:
\begin{lemma}
If $(x^{\alpha},\zeta_0,\zeta_{\alpha})$ is a solution of \cref{eqn:cogeo-tau-zeta} for $\tau \in [-\delta_0,0]$ with $x^{\alpha}|_{\tau=0}=0$ and $\zeta_0|_{\tau=0}=1$, then
\begin{equation}
x^{\alpha}(\tau) = -\frac{\kappa^{\al}}{2\kappa}\cdot \tau^2\log |\tau| - \left(\frac{\kappa}{2}  h^{\al\be}\zeta_{\be}|_{\tau=0}-\frac{\kappa^{\al}}{4\kappa}\right)\cdot \tau^2+o(\tau^2),
\end{equation}
where $\kappa$ and $\kappa^{\al}$ are evaluated at the origin $x^{\al}=0$ in the coordinate system for $\partial X$.
\end{lemma}
\begin{proof}
Recall that the solution is $\cC^1$ and that $\frac{dx^{\alpha}}{d\tau} = \rho h^{\al\be}(\zeta_{\beta}+A_{\beta})/\zeta_0$. Recall also that $\rho=-\kappa\tau +O(\tau^2)$. Since $\zeta_0=1+O(\tau)$ and $\zeta_{\be} = \zeta_{\be}|_{\tau=0} + O(\tau)$ it follows that 
\begin{equation}
    \frac{dx^{\alpha}}{d\tau} = -\frac{\kappa^{\al}}{\kappa}\cdot\tau\log|\tau| - \kappa h^{\al\be} \zeta_{\be}|_{\tau=0}\cdot \tau + O(\tau^2\log|\tau|). 
\end{equation}
Since $x^{\alpha}|_{\tau=0}=0$ it follows that
\begin{equation}
    x^{\alpha}(\tau) = -\frac{\kappa^{\al}}{2\kappa}\cdot \left(\tau^2\log|\tau|-\frac{1}{2}\tau^2\right)-\frac{\kappa}{2} h^{\al\be} \zeta_{\be}|_{\tau=0}\cdot \tau^2 + O(\tau^3\log|\tau|). 
\end{equation}
The result follows.
\end{proof}

From the lemma we see that every geodesic $\gamma(t)$ in $(X,g)$ that tends to $q$ as $t\to\infty$ is given near $q$ by \cref{eq:geodesic-asymp}, with $\mathcal{O}^{\al}$ given by \cref{eq:obstruction}. Moreover, when the geodesic (as unparametrized curve) is viewed as a solution to \cref{eqn:cogeo-tau-zeta} with $x^{\alpha}|_{\tau=0}=0$, $\zeta_0|_{\tau=0}=1$, the vector $u\in T_q\partial X$ appearing in \cref{eq:geodesic-asymp} is given by
\begin{equation}\label{eq:u-w-affine}
u^{\al} = -\frac{\kappa}{2}  h^{\al\be} \zeta_{\be}|_{\tau=0}+\frac{\kappa^{\al}}{4\kappa}.
\end{equation}
Since $\zeta_{\alpha}|_{\tau=0}$ may be prescribed freely when solving \cref{eqn:cogeo-tau-zeta} (near $\tau=0$) and this choice parametrizes the geodesics tending to $q$ (viewed as unparametrized curves), it follows from \cref{eq:u-w-affine} that one may equivalently parametrize the geodesics using the second-order coefficient $u\in T_q\partial X$ in the asymptotic expansion \cref{eq:geodesic-asymp}. This proves \cref{thm:shooting-geodesics-from-boundary}.

\appendix
\renewcommand*{\theHequation}{appendix.\Alph{section}.\arabic{equation}}

\section{Regularity results for ordinary differential equations}\label{app:reg}

In this section we prove some regularity results for a class of non-autonomous first-order systems of ordinary differential equations (ODE) that arise naturally in the study of geodesics in conformally compactified and related geometries. The results will follow from a straightforward modification of the standard argument for the smoothness of the flow of a time-dependent vector field, but we include it anyway for completeness. The only real catch is that, while in the smooth case the smoothness of the flow of a time-dependent vector field $V_{\tau}$ on a manifold $M$ follows from the smoothness of the flow of the vector field $\frac{\partial}{\partial \tau} + V_{\tau}$ on $\mathbb{R}\times M$ (with $\tau$ the coordinate for the $\mathbb{R}$-factor) and hence from a standard result on smooth dependence on initial conditions of solutions of autonomous first-order ODE (see, e.g., \cite[Appendix D]{Lee-smooth-manifolds}), when $V_{\tau}$ is no longer smooth in $\tau$ we must work directly with the corresponding non-autonomous system.

We therefore let $U$ be a domain in $\mathbb{R}^n$ and $I$ an interval and consider a first order system of the form
\begin{equation}\label{IVP}
  \begin{aligned}
    y'(\tau) &= V (\tau, y(\tau)),\\
    y(\tau_0)&=x,
  \end{aligned}
\end{equation} 
where $(\tau_0,x)\in I\times U$ and $V:I\times U \to \mathbb{R}^n$ is a continuous function  satisfying some additional regularity hypotheses that will be specified below. Since our considerations are local it will be convenient to take $I=\mathbb{R}$ and $U=\mathbb{R}^n$. For simplicity we will also assume that $V$ is bounded. In our intended applications we will always be able to assume that $V$ has been extended to a bounded function on all of $\mathbb{R}\times \mathbb{R}^n$ in a way that respects the regularity hypotheses.

The first additional condition we will impose on $V:\mathbb{R}\times \mathbb{R}^n \to \mathbb{R}^n$ is the standard hypothesis that it be \textit{uniformly Lipschitz} in the second factor, with Lipschitz constant $C>0$. By this we mean that there is $C>0$ such that 
\begin{equation}
\| V(\tau, y_1) - V(\tau,y_2)\| \leq C \|y_1 - y_2\|
\end{equation}
for all $\tau\in \mathbb{R}$ and $y_1,y_2 \in \mathbb{R}^n$. It then follows from the classical Picard-Lindel\"of theorem that for any $(\tau_0,x)\in \mathbb{R}\times \mathbb{R}^n$ there is a unique solution $y:(\tau_0-\delta,\tau_0+\delta) \to \mathbb{R}^n$ of \cref{IVP} provided $\delta < \frac{1}{C}$. Since for completeness we would also like to consider the initial time $\tau_0$ as a variable, we note that this implies that if $I_0$ is an open interval of length less than $\frac{1}{C}$ then for any $\tau_0 \in I_0$ and $x\in \mathbb{R}^n$ there is a unique solution $y:I_0 \to \mathbb{R}^n$ of \cref{IVP}. Fixing such an interval $I_0$ it follows that there is a map $\theta : I_0 \times I_0 \times \mathbb{R}^n\to \mathbb{R}^n$ such that the solution $y:I_0 \to \mathbb{R}^n$ of \cref{IVP} with initial condition $(\tau_0,x)\in I_0\times \mathbb{R}^n$ is given by 
\begin{equation}
    y(\tau) = \theta(\tau, \tau_0,x).
\end{equation}
We will refer to the map $\theta : I_0 \times I_0 \times \mathbb{R}^n\to \mathbb{R}^n$ as the \textit{local flow map} of $V$. For fixed $(\tau,\tau_0)\in I_0\times I_0$ we also define the map $\mathrm{Fl}_{\tau,\tau_0} : \mathbb{R}^n\to \mathbb{R}^n$ by
\begin{equation}
    \mathrm{Fl}_{\tau,\tau_0}(x) = \theta(\tau, \tau_0,x).
\end{equation}
We are interested in the regularity of the maps $\theta$ and $\mathrm{Fl}_{\tau,\tau_0}$ under additional regularity hypotheses on $V$. 

\subsection{Basic properties of the local flow map}

Note that, under the present hypotheses, for fixed $(\tau_0,x)\in I_0\times \mathbb{R}^n$ the local flow map $\theta(\tau, \tau_0,x)$ is clearly $\cC^1$ in $\tau$. Importantly, the local flow map $\theta$ also inherits a uniform Lipschitz property (in the third argument) from the uniform Lipschitz condition on $V$ (in the second argument). To see this, fix any $\tau_0 \in I_0$ and $x_1,x_2 \in \mathbb{R}^n$ and let $y_i(\tau) = \theta(\tau,\tau_0, x_i)$ for $i=1,2$. Then the Lipschitz property of $V$ implies that 
\begin{equation}
    \| y_1'(\tau) - y_2'(\tau)\| = \| V(\tau,y_1(\tau)) - V(\tau,y_2(\tau)) \| \leq C  \| y_1(\tau) - y_2(\tau)\|
\end{equation}
for all $\tau \in I_0$, and hence $\| y_1(\tau) - y_2(\tau)\|\leq e^{C|\tau-\tau_0|}\| x_1 - x_2\|$. In particular, since $|\tau - \tau_0|< \frac{1}{C}$ it follows that for any $\tau,\tau_0 \in I_0$ we have
\begin{equation}\label{ineq:partial-Lip-for-theta}
    \| \theta(\tau,\tau_0, x_1) - \theta(\tau,\tau_0, x_2)\| \leq e \| x_1 - x_2\|,
\end{equation}
for all $x_1,x_2 \in \mathbb{R}^n$.

Using \cref{ineq:partial-Lip-for-theta} one can then easily show that $\theta : I_0 \times I_0 \times \mathbb{R}^n\to \mathbb{R}^n$ is continuous. To see this, note that
\begin{equation}\label{eq:int-eq-for-theta}
    \theta(\tau,\tau_0,x) = x + \int_{\tau_0}^{\tau} V(s,\theta(s,\tau_0,x))\,ds
\end{equation}
so that 
\begin{equation}\label{eq:theta-diff}
\begin{aligned}
    \theta(\tilde{\tau},\tilde{\tau}_0,\tilde{x}) - \theta(\tau,\tau_0,x) &= (\tilde{x} - x) + \int_{\tau_0}^{\tau} \big( V(s,\theta(s,\tilde{\tau}_0,\tilde{x})) - V(s,\theta(s,\tau_0,x))\big) \,ds\\
    &\phantom{=} \qquad+ \int_{\tilde{\tau}_0}^{\tau_0} V(s,\theta(s,\tilde{\tau}_0,\tilde{x}))\,ds + \int_{\tau}^{\tilde{\tau}} V(s,\theta(s,\tilde{\tau}_0,\tilde{x})) \,ds. 
\end{aligned}
\end{equation}
Viewing $(\tau,\tau_0,x)$ as fixed, the only term on the right-hand side of the above display that does not obviously go to zero as $(\tilde{\tau},\tilde{\tau}_0,\tilde{x})\to (\tau,\tau_0,x)$ is the first integral. To estimate this term we note that setting $y(s) = \theta(s,\tau_0,x)$ we have
\begin{equation}
    \theta(s,\tau_0,x) = \theta(s,\tilde{\tau}_0,y(\tilde{\tau}_0)).
\end{equation}
Using the Lipschitz property of $V$ in the second argument and \cref{ineq:partial-Lip-for-theta} we therefore have (assuming for simplicity that $\tau>\tau_0$)
\begin{equation}
\begin{aligned}
    &\left\|\int_{\tau_0}^{\tau}  \big( V(s,\theta(s,\tilde{\tau}_0,\tilde{x})) - V(s,\theta(s,\tau_0,x))\big) \,ds \right\| \\ & \qquad \qquad \qquad \leq C|\tau - \tau_0|  \sup_{s\in[\tau_0,\tau]} \|\theta(s,\tilde{\tau}_0,\tilde{x})-\theta(s,\tau_0,x) \|\\
    & \qquad \qquad \qquad \leq \sup_{s\in[\tau_0,\tau]} \|\theta(s,\tilde{\tau}_0,\tilde{x})-\theta(s,\tau_0,x) \|   \\
    & \qquad \qquad \qquad= \sup_{s\in[\tau_0,\tau]} \|\theta(s,\tilde{\tau}_0,\tilde{x})-\theta(s,\tilde{\tau}_0,y(\tilde{\tau}_0)) \| \\
    & \qquad \qquad \qquad\leq e\| \tilde{x} - y(\tilde{\tau}_0)\|\\
    & \qquad \qquad \qquad\leq e\| \tilde{x} - x\| + e\| y(\tilde{\tau}_0) - x\|.
\end{aligned}    
\end{equation}
Since $y(s) = \theta(s,\tau_0,x)$ is continuous, it follows easily that the right-hand side of \cref{eq:theta-diff} goes to zero as  $(\tilde{\tau},\tilde{\tau}_0,\tilde{x})\to (\tau,\tau_0,x)$. Hence $\theta$ is continuous.

Note that, since  $\frac{\partial\theta}{\partial \tau}(\tau,\tau_0,x) = V(\tau,\theta(\tau,\tau_0,x))$, the continuity of $\theta$ implies that $\frac{\partial\theta}{\partial \tau}$ is continuous on $I_0\times I_0 \times\mathbb{R}^n$ (not merely on $I_0$ for each fixed $(\tau_0,x)\in I_0  \times\mathbb{R}^n$).

\subsection{A regularity theorem}

In our intended applications $V(\tau,y)$ will have a given regularity that is preserved under partial differentiation with respect to the $y$-variables, but not with respect to $\tau$. Taking our basic level of regularity in $(\tau,y)$ to be continuity, we therefore consider the regularity of the maps $\theta$ and $\mathrm{Fl}_{\tau,\tau_0}$ under the additional hypothesis that the $y$-partial derivatives of $V(\tau,y)$ of order $k$ exist and are everywhere continuous in the variables $(\tau,y)$. 

\begin{theorem}\label{thm:appendix-ode-reg-Ck}
Let $V:\mathbb{R}\times \mathbb{R}^n \to \mathbb{R}^n$ be a continuous function that is uniformly Lipschitz in the second factor, with Lipschitz constant $C>0$. Suppose in addition that the $k^{th}$ order partial derivatives of $V$ with respect to the coordinates of the second factor exist and are continuous on $\mathbb{R}\times \mathbb{R}^n$ $(k\geq 1)$. Let $I_0$ be an open interval of length less than $\frac{1}{C}$. Then
\begin{itemize}
\item[(i)] the local flow map $\theta : I_0 \times I_0 \times \mathbb{R}^n\to \mathbb{R}^n$ is $\cC^1$;
\item[(ii)] the $k^{th}$ order partial derivatives of $\theta$  with respect to the coordinates of the third factor exist and are continuous on $I_0 \times I_0 \times \mathbb{R}^n$;
\item[(iii)] for fixed $(\tau,\tau_0)\in I_0\times I_0$ the map $\mathrm{Fl}_{\tau,\tau_0} : \mathbb{R}^n\to \mathbb{R}^n$ is $\cC^k$.
\end{itemize}
\end{theorem}
\begin{remark} 
Note that, although \cref{thm:appendix-ode-reg-Ck} was stated for finite $k$, the case where $k=\infty$ clearly follows from the finite regularity case. 
\end{remark}

The key point here is that the regularity condition we imposed in the dependent variables holds uniformly in the independent variable, and this property is passed on to the local solution maps. 

We will prove \cref{thm:appendix-ode-reg-Ck} in stages in the subsections that follow. The proof is a straightforward modification of the standard argument for the case when $V$ is $\cC^k$ in both variables (see, e.g., \cite[Theorem D.5]{Lee-smooth-manifolds}). The catch, as we have mentioned, is that since the regularity in $\tau$ and $y$ is now different we cannot reduce to the time-independent (i.e.\ autonomous) case.

\subsection{Existence and continuity of first partials}

In this subsection we prove \cref{thm:appendix-ode-reg-Ck} for $k=1$.  Since (iii) clearly follows from (ii), it suffices to prove (i) and (ii). 

Suppose that $V$ satisfies the hypotheses of \cref{thm:appendix-ode-reg-Ck} for $k=1$. That $\frac{\partial\theta}{\partial\tau}$ exists and is continuous on $I_0 \times I_0 \times \mathbb{R}^n$ has already been observed. We therefore consider the partial derivative of $\theta$ with respect to the $y$-coordinates (and afterward the $\tau_0$ coordinate). The key observation here is that when $\tau=\tau_0$, $\theta(\tau,\tau_0,x)=x$ so that $\frac{\partial \theta^i}{\partial x^j}(\tau,\tau_0,x) = \delta_j^i$. We use this fact together with the differential equation $\frac{\partial\theta}{\partial \tau}(\tau,\tau_0,x) = V(\tau,\theta(\tau,\tau_0,x))$ satisfied by $\theta$ to gain an understanding of $\frac{\partial \theta^i}{\partial x^j}$ for general $(\tau, \tau_0)$. 

Fix any closed interval $\oJ_0\subset I_0$ and any compact subset $\oU_0 \subset \mathbb{R}^n$ and consider the $n\times n$ matrix-valued map $\Delta_h : \oJ_0\times \oJ_0\times \oU_0 \to \mathrm{M}_{n\times n}(\mathbb{R})$ defined by
\begin{equation}
    (\Delta_{h})^{i}_{j}(\tau,\tau_0,x) = \frac{\theta^i(\tau,\tau_0,x+he_j) - \theta^i(\tau,\tau_0,x)}{h},
\end{equation}
where $e_j$ is the $j^{th}$ standard basis vector in $\mathbb{R}^n$. Note that by \cref{ineq:partial-Lip-for-theta} we have 
\begin{equation}
    |(\Delta_{h})^{i}_{j}(\tau,\tau_0,x)| \leq e^{C|\tau-\tau_0|}< e
\end{equation} 
for each $i$ and $j$.
From the differential equation satisfied by $\theta$ it follows that 
\begin{equation}
    \frac{\partial}{\partial\tau}(\Delta_{h})^{i}_{j}(\tau,\tau_0,x) = \frac{1}{h}\left(V^i(\tau,\theta(\tau,\tau_0,x+he_j)) - V^i(\tau,\theta(\tau,\tau_0,x))  \right).
\end{equation}

For each  $(\tau,\tau_0,x)\in \oJ_0\times \oJ_0\times \oU_0$ (and each $i$ and $j$), consider the $\cC^1$ function
\begin{equation}
u(s) = V^i(\tau,(1-s)\theta(\tau,\tau_0,x)+s\theta(\tau,\tau_0,x+he_j)).
\end{equation}
By the mean value theorem there is a point $c\in (0,1)$ such that $u(1)-u(0) = u'(c)$. Hence, if $\vartheta = (1-c)\theta(\tau,\tau_0,x)+c\theta(\tau,\tau_0,x+he_j)$ then
\begin{equation}
    \frac{\partial}{\partial\tau}(\Delta_{h})^{i}_{j}(\tau,\tau_0,x) =  \sum_{k=1}^n \frac{\partial V^i}{\partial y^k} (\tau,\vartheta)\cdot (\Delta_{h})^{k}_{j}(\tau,\tau_0,x).
\end{equation}
Note that $\vartheta$ depends on $(\tau,\tau_0,x)\in \oJ_0\times \oJ_0\times \oU_0$ (and $i$ and $j$), but, since it lies on the line segment between $\theta(\tau,\tau_0,x)$ and $\theta(\tau,\tau_0,x+he_j)$ and $\theta$ is continuous, $\vartheta$ tends to $\theta(\tau,\tau_0,x)$ as $h\to 0$ (uniformly in $(\tau,\tau_0,x)$, since $ \oJ_0\times \oJ_0\times \oU_0$ is compact). If $\tilde{h}$ is another (small, nonzero) real number, we define $\tilde{\vartheta}$ similarly. Letting $Q_{h,\tilde{h}} = (\Delta_{h})(\tau,\tau_0,x) - (\Delta_{\tilde{h}})(\tau,\tau_0,x)$ it follows that
\begin{equation}\label{eq:del-t-Qhh}
\begin{aligned}
    \frac{\partial}{\partial \tau} (Q_{h,\tilde{h}})_j^i &= \frac{\partial}{\partial\tau} \left( (\Delta_{h})^{i}_{j}(\tau,\tau_0,x) - (\Delta_{\tilde{h}})^{i}_{j}(\tau,\tau_0,x) \right)\\
    &= \sum_{k=1}^n \frac{\partial V^i}{\partial y^k} (\tau,\vartheta)\cdot (\Delta_{h})^{k}_{j}(\tau,\tau_0,x) - \sum_{k=1}^n \frac{\partial V^i}{\partial y^k} (\tau,\tilde{\vartheta})\cdot  (\Delta_{\tilde{h}})^{k}_{j}(\tau,\tau_0,x) \\
    &= \sum_{k=1}^n \frac{\partial V^i}{\partial y^k} (\tau,\vartheta)\cdot \left((\Delta_{h})^{k}_{j}(\tau,\tau_0,x) - (\Delta_{\tilde{h}})^{k}_{j}(\tau,\tau_0,x)\right) \\
    &\phantom{=}\quad + \sum_{k=1}^n \left(\frac{\partial V^i}{\partial y^k} (\tau,\vartheta) - \frac{\partial V^i}{\partial y^k} (\tau,\tilde{\vartheta}) \right)\cdot  (\Delta_{\tilde{h}})^{k}_{j}(\tau,\tau_0,x).
    \end{aligned}
\end{equation}

Now, since $\theta$ is continuous, the image of $\oJ_0\times \oJ_0\times \oU_0$ is a compact set. Moreover, requiring that $|h|<1$, $x+he_j$  will always lie in the compact set $\oU_1 = \oU_0 + \overline{\mathbb{B}^n}$ and hence $\vartheta$ will always lie in the convex hull $\oW_{\!1}$ of the image of $\oJ_0\times \oJ_0\times \oU_1$ under $\theta$, which is also compact. Since the (first order) $y$-partial derivatives of $V$ are continuous in $(\tau,y)$, they are bounded on $\oJ_0\times \oW_{\!1}$. 
It therefore follows from \cref{eq:del-t-Qhh} that there is a constant $E>0$ such that 
\begin{equation}\label{eq:del-t-Qhh-est}
    \left|  \frac{\partial}{\partial \tau} (Q_{h,\tilde{h}})_j^i \right| \leq E  \left\| Q_{h,\tilde{h}}\right\| + e \sum_{k=1}^n \left\| \left(\frac{\partial V^i}{\partial y^k} (\tau,\vartheta) - \frac{\partial V^i}{\partial y^k} (\tau,\tilde{\vartheta}) \right) \right\|,
\end{equation}
provided $|h|<1$ and $|\tilde{h}|<1$. Since the $y$-partial derivatives of $V$ are also uniformly continuous on $\oJ_0\times \oW_{\!1}$, for any $\epsilon >0$ there is a $\delta > 0$ such that 
\begin{equation} \label{ineq:unif-cty-del-y-V}
\left| \frac{\partial V^i}{\partial y^k} (\tau,\vartheta) - \frac{\partial V^i}{\partial y^k} (\tau,\tilde{\vartheta}) \right| < \epsilon, \quad \text{for all }i \text{ and } k, 
\end{equation}
provided $|\tilde{\vartheta}-\vartheta |<\delta$. Now, for each $i$ and $j$, the $\vartheta$ and $\tilde{\vartheta}$ in \cref{eq:del-t-Qhh} and \cref{eq:del-t-Qhh-est} can be viewed as functions of $(\tau,\tau_0,x)$ that converge uniformly to $\theta(\tau,\tau_0,x)$ on $\oJ_0\times \oJ_0\times \oU_0$ as $h\to 0$ and $\tilde{h}\to 0$, respectively. Hence, for any $\epsilon >0$, by taking $|h|$ and $|\tilde{h}|$ sufficiently small we may ensure that \cref{ineq:unif-cty-del-y-V} holds (for all relevant indices) and hence that 
\begin{equation}\label{eq:del-t-Qhh-est-eps}
    \left\|  \frac{\partial }{\partial \tau} Q_{h,\tilde{h}}\right\| \leq nE  \left\| Q_{h,\tilde{h}}\right\| +  \epsilon n^2e.
\end{equation}

Note that when $\tau=\tau_0$ we have $Q_{h,\tilde{h}} = (\Delta_{h})(\tau_0,\tau_0,x) - (\Delta_{\tilde{h}})(\tau_0,\tau_0,x) = \mathrm{Id} -\mathrm{Id}=0$. For general $(\tau,\tau_0)$ it therefore follows by an ODE comparison argument that for any $\epsilon >0$ one has
\begin{equation}
\begin{aligned}
    \|Q_{h,\tilde{h}}\| = \|(\Delta_{h})(\tau,\tau_0,x) - (\Delta_{\tilde{h}})(\tau,\tau_0,x)\| &\leq \frac{\epsilon n e}{E}\left(e^{nE|\tau-\tau_0|}-1 \right) \\ &\leq \frac{\epsilon n e}{E}\left(e^{nE/C}-1 \right),
\end{aligned}
\end{equation}
provided $h$ and $\tilde{h}$ are sufficiently close to $0$. It follows that for any sequence of nonzero numbers $h_k$ that tends to zero, the corresponding sequence of matrix-valued functions $\Delta_{h_k}$ is uniformly Cauchy and hence uniformly convergent. Since the limit is clearly independent of the choice of $h_k$ it follows that $\lim_{h\to 0} \Delta_h$ exists (and is continuous). In particular, $\frac{\partial \theta^i}{\partial x^j}$ exists and is continuous (on  $\oJ_0\times \oJ_0\times \oU_0$) for each $i$ and $j$. Since compact sets of the form $\oJ_0\times \oJ_0\times \oU_0$ exhaust $I_0\times I_0\times \mathbb{R}^n$, the first order $x$-partial derivatives of $\theta$ exist and are continuous on $I_0\times I_0\times \mathbb{R}^n$.

It remains to show that the same holds for the partial derivative of $\theta$ with respect to $\tau_0$ (i.e. the second factor). To see this, we note that $\mathrm{Fl}_{\tau,\tau_0}\circ \mathrm{Fl}_{\tau_0,\tau} : \mathbb{R}^n \to \mathbb{R}^n$ is the identity map. Thus $\mathrm{Fl}_{\tau,\tau_0} = \mathrm{Fl}_{\tau_0,\tau}^{-1}$, and hence the existence and continuity of the $\tau_0$-partial derivative of $\theta(\tau,\tau_0,x) = \mathrm{Fl}_{\tau_0,\tau}^{-1} (x)$ follows from the existence and continuity of the partial derivatives of $\theta$ in the other two arguments (and the inverse function theorem). Moreover, differentiating the expression $\theta(\tau, \tau_0, \theta(\tau_0,\tau,x)) =x$  with respect to $\tau_0$ (and setting $x_0=\theta(\tau_0,\tau,x)$, so that $x= \theta(\tau,\tau_0,x_0)$) we find that
\begin{equation}
    \frac{\partial \theta^i}{\partial \tau_0} (\tau,\tau_0,x_0) = -\sum_{k=1}^n \frac{\partial \theta^i}{\partial x^k} (\tau,\tau_0,x_0)\cdot \frac{\partial \theta^k}{\partial \tau} (\tau_0,\tau,\theta(\tau,\tau_0,x_0)).
\end{equation}

This proves \cref{thm:appendix-ode-reg-Ck} in the case where $k=1$.

\subsection{Higher regularity}

It remains to prove \cref{thm:appendix-ode-reg-Ck}(ii) for $k>1$ (part (iii) then follows). We do this by induction on $k$. Let $k\geq 1$ and suppose that \cref{thm:appendix-ode-reg-Ck}(ii) is known to hold for this $k$. Suppose that $V$ satisfies the hypotheses of \cref{thm:appendix-ode-reg-Ck} with $k$ replaced by $k+1$.

By the previous subsection, the local flow map $\theta : I_0\times I_0 \times \mathbb{R}^n\to \mathbb{R}^n$ of $V$ is $\cC^1$. Moreover, from \cref{eq:int-eq-for-theta} we have
\begin{equation}
    \frac{\partial \theta^i}{\partial x^j}(\tau,\tau_0,x) = \delta^i_j + \sum_{k=1}^n \int_{\tau_0}^{\tau} \frac{\partial V^i}{\partial y^k}(s,\theta(s,\tau_0,x))\cdot \frac{\partial \theta^k}{\partial x^j}(s,\tau_0,x)\,ds.
\end{equation}
It follows that for each $i$ and $j$ the $\tau$-partial derivative of the function $\frac{\partial \theta^i}{\partial x^j}$ exists (and is continuous in $(\tau,\tau_0,x)$) and satisfies
\begin{equation}
   \frac{\partial}{\partial \tau}\frac{\partial \theta^i}{\partial x^j}(\tau,\tau_0,x) = \sum_{k=1}^n  \frac{\partial V^i}{\partial y^k}(\tau,\theta(\tau,\tau_0,x))\cdot \frac{\partial \theta^k}{\partial x^j}(\tau,\tau_0,x).
\end{equation}
We therefore augment the dependent variable $y=y(\tau)$ in our initial value problem \cref{IVP} with an $n\times n$ matrix-valued function $Y=Y(\tau)$ and consider the following initial value problem for $(y,Y)$: 
\begin{equation}\label{eq:augmented-IVP}
\begin{aligned}
     y'(\tau) &= V(\tau,y(\tau)) \qquad \qquad &y(\tau_0)&=y_0 \\
     Y'(\tau) &= D_yV(\tau,y(\tau))\cdot   Y(\tau) \qquad& Y(\tau_0) &= Y_0.
\end{aligned}
\end{equation}
After fixing a compact subset $\oU_0\subset\mathbb{R}^n$, the case
$k=1$ shows that the relevant pairs $(\theta,D_x\theta)$ range over
a compact set on compact time intervals. Cutting off the augmented
vector field
\[
F(\tau,y,Y)=\left(V(\tau,y),D_yV(\tau,y)Y\right)
\]
outside a compact neighborhood of the corresponding set in
$(\tau,y,Y)$-space ensures that the localized augmented system
satisfies the hypotheses of \cref{thm:appendix-ode-reg-Ck}, without
changing the solutions under consideration.
It follows that for any sufficiently small open subinterval $I_1$ of $I_0$ there is a well-defined local flow
\begin{equation}
    (\Theta,\Psi): I_1\times I_1\times \bigl(\mathbb{R}^n \times \mathrm{M}_{n\times n}(\mathbb{R}) \bigr) \to \mathbb{R}^n\times \mathrm{M}_{n\times n}(\mathbb{R})
\end{equation}
for the localized version of the system \cref{eq:augmented-IVP}, and that $\Theta$ and $\Psi$ have $k^{th}$ order partial derivatives with respect to the third factor. Since $y(\tau) = \theta(\tau,\tau_0,x)$ and $Y(\tau)= D_x \theta(\tau,\tau_0,x)$ solve this system with initial conditions $y_0=x$ and $Y_0 = \mathrm{Id}$, we have
\begin{equation}
    \theta(\tau,\tau_0,x) = \Theta(\tau,\tau_0,(x,\mathrm{Id})) \quad \text{ and }\quad D_x \theta(\tau,\tau_0,x) = \Psi(\tau,\tau_0,(x,\mathrm{Id}))
\end{equation}
when  $(\tau,\tau_0,x)\in I_1\times I_1\times \oU_0$. It follows that the first order $x$-partial derivatives $\frac{\partial \theta^i}{\partial x^j}$ have $k^{th}$ order $x$-partial derivatives that are continuous on $I_1\times I_1\times \oU_0$. 
For arbitrary $\tau,\tau_0\in I_0$, choose
$\tau_0=s_0,\ldots,s_m=\tau$ so that each consecutive pair lies in
such a short interval. The flow identity
\[
\mathrm{Fl}_{\tau,\tau_0}
 =\mathrm{Fl}_{s_m,s_{m-1}}\circ\cdots\circ
  \mathrm{Fl}_{s_1,s_0}
\]
gives the same conclusion by finite composition; joint continuity
follows locally by fixing the intermediate times.
Thus $\theta$ has $x$-partial derivatives of order $k+1$ that are continuous on $I_0\times I_0\times \mathbb{R}^n$. This proves the inductive step.

It follows by induction that \cref{thm:appendix-ode-reg-Ck}(ii) holds for all $k\geq 1$. This concludes the proof of \cref{thm:appendix-ode-reg-Ck}.

\address{
Department of Mathematics, Oklahoma State University,\\
Stillwater, OK 74078-5061, USA\\
\email{sean.curry@okstate.edu}
}

\address{
Department of Mathematics, UC San Diego,\\
La Jolla, CA 92093-0021, USA\\
\email{anandi@ucsd.edu}
}

\end{document}